\documentclass[12pt,letterpaper]{scrartcl}
\usepackage[utf8]{inputenc}
\usepackage{tikz}
\usepackage{amsthm,amsfonts, mathtools, amssymb, bbm,amsmath,mathdots}
\usetikzlibrary{matrix}
\usetikzlibrary{cd}
\usepackage[square,numbers,sort&compress]{natbib}
\usepackage{graphicx}
\usepackage{transparent}
\usepackage[percent]{overpic}
\usepackage{caption}
\usepackage{subcaption}
\usepackage{xcolor,color,soul} 
\usepackage{enumitem}
\usepackage{physics}
\usepackage{qtree}
\usepackage[export]{adjustbox} 
\usepackage{appendix}
\definecolor{Myblue}{rgb}{0,0,0.6}  
\usepackage[colorlinks,citecolor=Myblue,linkcolor=Myblue,urlcolor=Myblue,pdfpagemode=None]{hyperref}
\usepackage[totalwidth=460pt, totalheight=650pt]{geometry}

\newcommand{\boxpic}[3]{
	\begin{tikzpicture}[baseline={([yshift=-.5ex]current bounding box.center)}]
	\node at (0,0) {\begin{overpic}[scale=#1]{#2}#3\end{overpic}};
	\end{tikzpicture}
}

\theoremstyle{definition}
\newtheorem{definition}{Definition}
\newtheorem{theorem}[definition]{Theorem}
\newtheorem{lemma}[definition]{Lemma}
\newtheorem{proposition}[definition]{Proposition}
\newtheorem{example}[definition]{Example}
\newtheorem{corollary}[definition]{Corollary}
\newtheorem{remark}[definition]{Remark}

\numberwithin{definition}{section}
\numberwithin{equation}{section}
\numberwithin{figure}{section}

\newcommand{\id}[1]{\operatorname{id}_{#1}}
\newcommand{\unit}{\mathbbm{1}} 
\newcommand{\catname}[1]{{\mathcal{#1}}} 

\newcommand{\Hom}{\mathrm{Hom}} 
\newcommand{\End}{\mathrm{End}}
\newcommand{\Aut}{\mathrm{Aut}}
\newcommand{\Mod}{\mathrm{Mod}}

\newcommand{\doublecat}{\catname C ^\pm}

\newcommand{\ibar}{\overline{i}}
\newcommand{\jbar}{\overline{j}}
\newcommand{\kbar}{\overline{k}}

\newcommand{\ie}{i.e.\,}

\allowdisplaybreaks

\begin{document}

\title{Mapping class group representations and Morita classes of algebras}

\author{
	Iordanis Romaidis \qquad
	Ingo Runkel\\[0.5cm]
	\normalsize{\texttt{\href{mailto:iordanis.romaidis@uni-hamburg.de}{iordanis.romaidis@uni-hamburg.de}}} \\  
	\normalsize{\texttt{\href{mailto:ingo.runkel@uni-hamburg.de}{ingo.runkel@uni-hamburg.de}}}\\[0.5cm]
	\normalsize\slshape Fachbereich Mathematik, Universit\"{a}t Hamburg,\\
	\normalsize\slshape Bundesstra{\ss}e 55, 20146 Hamburg, Germany
	}

\date{}

\maketitle

\begin{abstract}
A modular fusion category $\mathcal{C}$ allows one to define projective representations of the mapping class groups of closed surfaces of any genus. We show that if all these representations are irreducible, then $\mathcal{C}$ has a unique Morita-class of simple non-degenerate algebras, namely that of the tensor unit. This improves on a result by Andersen and Fjelstad, albeit under stronger assumptions. One motivation to look at this problem comes from questions in three-dimensional quantum gravity.
\end{abstract}

\newpage

\setcounter{tocdepth}{1}
\tableofcontents

\section{Introduction}

Let $\catname C$ be a modular fusion category, that is, a finitely semisimple ribbon category with simple tensor unit whose braiding is non-degenerate. Famously, such categories give rise to three-dimensional topological quantum field theories \cite{RT2,Tur}, and consequently also to (projective) representations of surface mapping class groups. 

In more detail, let $I$ denote a choice of representatives of the isomorphism classes of simple objects in $\catname C$ and write $L = \bigoplus_{i \in I} i \otimes i^*$.
Denote by $\Mod_g$ the mapping class group of a closed genus-$g$ surface without marked points. Then $\Mod_g$ acts projectively on the Hom-space
$$
V_g^{\catname C} :=
    \catname C(\unit,L^{\otimes g})~,
$$
and we recall this action in Section~\ref{sec:MCG-action}.

An algebra $A \in \catname C$ is called \textit{non-degenerate} if its trace pairing is non-degenerate, see Section~\ref{sec:invariants} for details. Non-degenerate algebras carry a symmetric Frobenius structure. An algebra is called \textit{simple} if it is simple as a bimodule over itself. 
Two algebras $A,B$ are \textit{Morita-equivalent} if there are bimodules 
${}_AX_B$ and ${}_BY_A$ in $\catname C$ such that $X \otimes_B Y \cong A$ and $Y \otimes_A X \cong B$ as bimodules. 

Our main result is (see Theorem~\ref{thm:maintheorem}):

\begin{theorem}\label{thm:main-intro}
Let $\mathcal{C}$ be a modular fusion category over an algebraically closed field of characteristic zero. If the projective mapping class group representations $V_g^{\catname C}$ are irreducible for all $g \ge 0$, then every simple non-degenerate algebra in $\mathcal{C}$ is Morita-equivalent to the tensor unit.
\end{theorem}

Suppose now that the modular fusion category $\catname C$ is defined over $\mathbb{C}$. In this case, $\catname C$ is called \textit{pseudo-unitary} if  all simple objects have positive quantum dimension, see \cite[Sec.\,8]{ENO} for details. Combining Theorem~\ref{thm:main-intro} with results on the existence of module traces in \cite{Schau}, it turns out we can drop the non-degeneracy condition (see Corollary~\ref{cor:main-pseudounitary}):

\begin{corollary}
Suppose that in addition to the hypotheses in Theorem~\ref{thm:main-intro}, $\catname C$ is defined over $\mathbb{C}$ and pseudo-unitary. Then all simple algebras in $\catname C$ are Morita-equivalent to the tensor unit.
\end{corollary}

A result closely related to Theorem~\ref{thm:main-intro} is proven in \cite{Andersen:2008}. There it is shown that if there is a $g \ge 1$ such that $V_g^{\catname C}$ is irreducible, then for every simple non-degenerate algebra $A$, its full centre $Z(A) \in \catname C \boxtimes \catname{C}^\mathrm{rev}$ has underlying object $\bigoplus_{i \in I} i^\ast \times i$.
We recall the definition of the full centre in Section~\ref{sec:full-centre} and give a more detailed comparison to \cite{Andersen:2008} in Section~\ref{sec:mainthm}. Here we just note that under our stronger assumptions we can prove a stronger result, which in this language means that the full centres satisfy $Z(A) \cong Z(\unit)$ as algebras and not only as objects.
Our method of proof is quite different from that in \cite{Andersen:2008} and so may be of independent interest.

\medskip

Examples of $\catname C$ where all $V_g^{\catname C}$ are irreducible are when $\catname C$ is of Ising-type \cite{Jian:2019ubz,RR-prep}, and when it is given by $\catname C(sl(2),k)$ -- the modular fusion category for the affine Lie algebra $\widehat{sl}(2)$ at level $k$ -- when $k+2$ is prime \cite{Roberts:1999}. We will look at these in more detail in Example~\ref{ex:irred-or-not}, but it would certainly be good to have more examples at hand.

\begin{remark}
~
\begin{enumerate}
    \item 
The converse of the statement in Theorem~\ref{thm:main-intro} does not hold: $\catname C$ can have a unique Morita-class of simple non-degenerate algebras while e.g.\ $V^{\catname C}_{g=1}$ is reducible. In fact this is the typical situation. For example, for $\catname C(sl(2),k)$ with $k$ odd there is a unique such Morita-class \cite{Os}, but for $k+2$ odd and not prime or the square of a prime,
$V^{\catname C}_{g=1}$ is reducible, see Example~\ref{ex:irred-or-not}.
    
    \item In Theorem~\ref{thm:main-intro} it is enough to demand irreducibility of $V_g^{\catname C}$ for  $1 \le g \le 3N+2$, where $N$ is the length of the filtration of the adjoint subring of the Grothendieck ring of $\catname C$, see Remark~\ref{rem:maintheorem} and Section~\ref{sec:universalgroup} for details. A coarse bound for $N$ is the number of isomorphism classes of simple objects of $\catname C$, \ie $N \le |I|$.
    
\end{enumerate}
\end{remark}

A somewhat surprising motivation to look at irreducible mapping class group representations comes from quantum gravity in three dimensions. We summarise this in the next remark which can safely be skipped by readers less interested in speculations related to physics. Nonetheless, this is the reason why we started to study this problem.

\begin{remark}
~
\begin{enumerate}
    \item Euclidean AdS${}_3$ is topologically a solid torus. It has been argued in \cite{Maloney:2007ud} that the saddle-point approximation of the path integral of 3d quantum gravity includes a sum over geometries obtained by gluing the solid torus to its boundary by an element of $SL(2,\mathbb{Z})$, the mapping class group of the torus. By the AdS/CFT-correspondence one would expect 3d quantum gravity on AdS${}_3$ to be equivalent to a 2d conformal field theory  (or to an ensemble thereof) on its boundary, a 2-torus. One arrives at the following question: When does a sum over mapping class group orbits produce a consistent system of correlators of a 2d CFT? 
    In the context of rational 2d CFT, this has been analysed for genus 1 in \cite{Castro:2011zq} and for all genera for the Ising CFT in \cite{Jian:2019ubz}. The question whether one obtains a single 2d CFT or an ensemble has been investigated for WZW models in \cite{Meruliya:2021utr}. Comparing to the examples above, one finds 
    that e.g.\ for $SU(2)$-WZW models on the torus there is a single CFT
    in particular at levels
    $k$ with $k+2$ prime, where the mapping class group is known to act irreducibly on all $V^{\catname C}_g$. 
We will show  in \cite{RR-prep} that if the mapping class group orbits are finite (so that the sum is well-defined)
    and if the mapping class group representations are irreducible, then the sum produces a consistent system of rational 2d CFT correlators on surfaces of arbitrary genus and with insertion points.

    \item Morita classes of simple non-degenerate algebras in a modular fusion category $\catname C$ describe indecomposable surface defects in the 3d TQFT corresponding to $\catname C$ \cite{KS,FSV,CRS2}. In this context, Theorem~\ref{thm:main-intro} states that if all $V_g^{\catname C}$ are irreducible, then the corresponding 3d TQFT has no non-trivial surface defects. 
    Invertible surface defects are global symmetries of the 3d TQFT, and their absence ties in with a conjectural constraint on quantum gravity theories, namely that they should have no global symmetries, see \cite{Harlow:2018jwu} for a discussion in the context of AdS/CFT.
    
    Combining this with part 1, we see that, on the one hand, irreducibility of the $V_g^{\catname C}$ relates to consistency of the 2d CFT on the boundary and, on the other hand, to the absence of global symmetries of the 3d theory in the bulk.
    In fact, we obtain a stronger result, namely absence of all non-trivial surface defects, not just those related to global symmetries.
    We refer to \cite{RR-prep} for more details.
\end{enumerate}
\end{remark}

Before we start with the main part of the paper, let us mention the main ingredients in the proof.
They are the mapping class group actions obtained from 3d TQFT \cite{RT2,Tur} (Section~\ref{sec:MCG-action}), the invariants under this action obtained from non-degenerate algebras \cite{FRS1,KR1} (Section~\ref{sec:invariants}), the relation between Morita classes of algebras and their full centres \cite{KR,ENO} (Section~\ref{sec:full-centre}), and the universal grading group of a fusion category \cite{Gelaki:2006} (Section~\ref{sec:universalgroup}). The proof in Section~\ref{sec:mainthm} then works by reducing the difference between the algebra structures of the full centre of a given algebra and that of the tensor unit to a symmetric 2-cocycle on the universal grading group, which must be a coboundary.

\subsubsection*{Acknowledgements}

We would like to thank 
    Alexei Davydov,
    C\'esar Galindo,
    Vincentas Mulevi\v{c}ius,
    Timo Weigand
and
    especially Jens Fjelstad
for discussions and comments on a draft of this paper.
IRo is supported fully, and IRu partially, by the Deutsche
Forschungsgemeinschaft
(DFG, German Research Foundation) under Germany's Excellence Strategy - EXC 2121
``Quantum Universe'' - 390833306.

\subsubsection*{Conventions}

By a modular fusion category we mean a fusion category which is ribbon and whose braiding is non-degenerate in the sense that all transparent objects are isomorphic to a direct sum of tensor units. 
We refer to e.g.\ \cite[Sec.\,8.13]{EGNO} or \cite[Sec.\,4.5]{TuVire} for definitions and details.
Throughout this text, $\catname C$ will be a modular fusion category over an algebraically closed field $\mathbbm{k}$ of characteristic $0$. In order to simplify notation, we assume the monoidal structure of $\catname C$ to be strict.

\section{Mapping class group action on state spaces}\label{sec:MCG-action}

In this section we briefly review how to obtain projective representations of surface mapping class groups from a modular fusion category.

\medskip

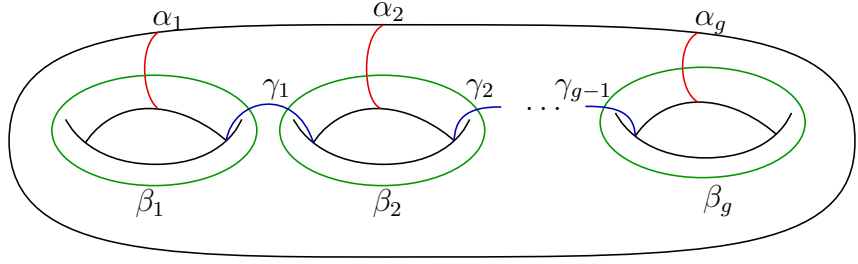
\begin{figure}[tb]
    \centering
    \begin{overpic}[scale = 0.7]{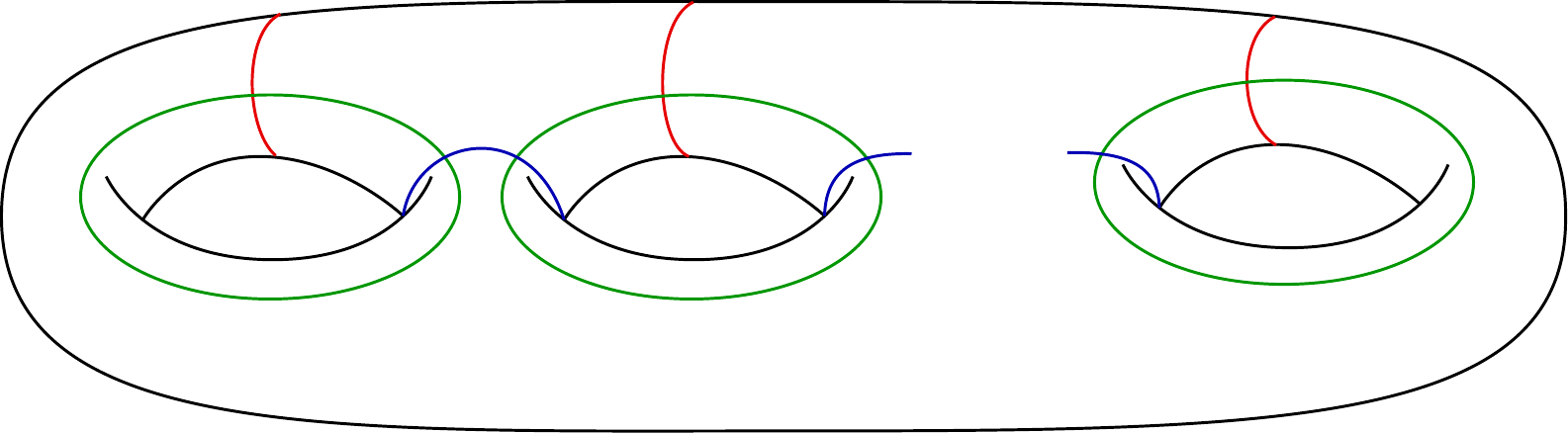}
        \put (17,27.5) {$\alpha_1$}
        \put (43,28.5) {$\alpha_2$}
        \put (81,27.5) {$\alpha_g$}
        \put (15,5.5) {$\beta_1$}
        \put (43,5.5) {$\beta_2$}
        \put (82,6) {$\beta_g$}
        \put (30,19.5) {$\gamma_1$}
        \put (54,19.5) {$\gamma_2$}
        \put (64.5,19.5) {$\gamma_{g-1}$}
        \put (61,17.5) {$\dots$}
    \end{overpic}
    \caption{Lickorish generators of $\Mod_g$.}
    \label{fig:mcggenerators}
\end{figure}

Let $\Sigma_g$ be a (smooth, compact, oriented, closed) surface of genus $g \ge 0$. We will denote by $\Mod_g = \Mod(\Sigma_g)$ the mapping class group, which is the group of isotopy classes of diffeomorphisms of $\Sigma_g$. 
Given a simple closed curve $\gamma$ on $\Sigma_g$, one can define the Dehn twist $T_\gamma$ as a mapping class in $\Mod_g$ \cite[Ch.\,3]{FM}. In fact, the mapping class group is finitely generated by such Dehn twists. An explicit set of generators, called the Lickorish generators, consists of Dehn twists along the curves shown in Figure~\ref{fig:mcggenerators}. Using these generators, one can define the so-called $S$-transformations $S_k := T_{\alpha_k}\circ T_{\beta_k}\circ T_{\alpha_k}$. Henceforth, we will use the alternative set of generators 
\begin{equation}\label{eq:MCG-gen}
\big\{\,T_{\alpha_1},\dots, T_{\alpha_g}, T_{\gamma_1},\dots,T_{\gamma_{g-1}},S_1,\dots,S_g\,\big\}~,
\end{equation}
where we replaced the generators $T_{\beta_k}$ by the corresponding $S$-transformations $S_k$.

\medskip

Given a modular fusion category (MFC) $\catname C$, the Reshetikhin-Turaev topological quantum field theory (RT-TQFT) for $\catname C$ gives rise to projective mapping class group representations \cite[Ch.\,IV.5]{Tur}. To describe these representations we first fix some notation. 

We will write $I$ for a set of representatives of the isomorphism classes of simple objects in $\catname C$, and we will assume that $\unit \in I$. Define the object
\begin{equation}\label{eq:L-def}
    L ~=~ \bigoplus_{i \in I} \, i \otimes i^*
    ~~\in\,\catname C
    ~.
\end{equation}
This object is used in the description of surgery in RT-TQFT. The quantum dimension of an object $U \in \catname C$ is denoted by $\dim_{\catname C}(U)$, and we abbreviate
\begin{equation}
    d_i = \dim_{\catname C}(i) ~~\text{for}~i \in I~,\quad 
    D = \sqrt{\dim_{\catname C} L} = \sqrt{{\textstyle \sum_{i \in I} (d_i)^2}}~.
\end{equation}
The choice of square root for $D$ does not matter here, it just changes the normalisation of one of the generators of the projective action below.

We will denote the projective representation of $\Mod_g$ by $V_g^{\catname C} \equiv V^{\catname C}(\Sigma_g)$. 
The underlying vector space is the Hom-space
\begin{equation}
    V_g^{\catname C} \equiv 
    V^{\catname C}(\Sigma_g) ~:=~ \catname C (\unit, L^{\otimes g}) ~.
\end{equation}
The Hom-space $V^{\catname C}(\Sigma_g)$ decomposes into the direct sum $\bigoplus_{i\in I^g} V_i$ where $i = (i_1,\dots,i_g)$ is a multi-index and $V_i := \catname C(\unit, i_1\otimes i_1^\ast \otimes \dots \otimes i_g\otimes i_g^\ast)$. 
The generators \eqref{eq:MCG-gen} act on $f_i \in V_i$ as follows:
\begin{align}\label{eq:MCG-gen-action}
T_{\alpha_k}(f_i) = \theta_{i_k} &\boxpic{0.7}{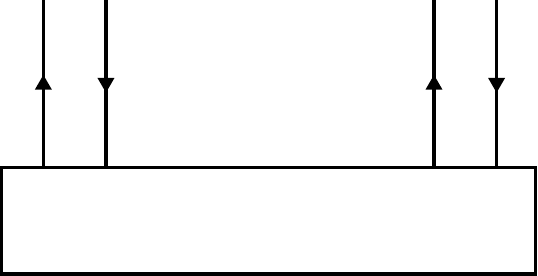}{
        \put (6,53) {$i_1$}
        \put (18,53) {$i_1^\ast$}
        \put (78,53) {$i_g$}
        \put (91,53) {$i_g^\ast$}
        \put (45,7.5) {$f_i$}
        \put (44,32.5) {$\cdots$}
        }
&
T_{\gamma_k}(f_i) =
    \boxpic{0.7}{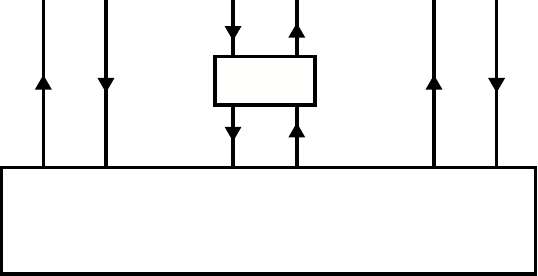}{
        \put (6,53) {$i_1$}
        \put (18,53) {$i_1^\ast$}
        \put (41,53) {$i_k^\ast$}
        \put (52.5,53) {$i_{k+1}$}
        \put (78.5,53) {$i_g$}
        \put (90.5,53) {$i_g^\ast$}
        \put (45,7.5) {$f_i$}
        \put (46.5,32.8) {\small$\theta$}
        \put (25,32.5) {$\cdots$}
        \put (63,32.5) {$\cdots$}
        }
\nonumber\\[2ex]
S_k(f_i)=\bigoplus_{j\in I} \frac{d_j}{D}
    &\boxpic{0.7}{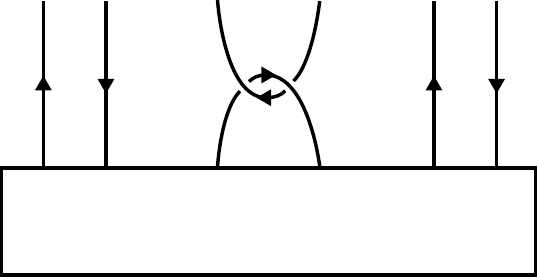}{
        \put (6,53) {$i_1$}
        \put (18,53) {$i_1^\ast$}
        \put (37,53) {$j$}
        \put (55,53) {$j^\ast$}
        \put (78,53) {$i_g$}
        \put (90,53) {$i_g^\ast$}
        \put (45,7.5) {$f_i$}
        \put (32,25) {$i_k$}
        \put (25,32.5) {$\cdots$}
        \put (63,32.5) {$\cdots$}
    }
\end{align}
Here we use string diagram notations for morphisms in $\catname C$. Our diagrams are read from bottom to top and our conventions for dualities, braiding and twist match those in \cite[Sec.\,2.1]{FRS1}. The constant $\theta_{i_k} \in \mathbbm{k}^\times$ is the value of the ribbon twist on the simple object $i_k$.

The expressions for the generators $T_{\alpha_k}$ and $S_k$ are given e.g.\ in~\cite[Def. 3.1.15]{BK}.

\begin{figure}[tb]
    \centering
    \begin{overpic}[scale = 0.7]{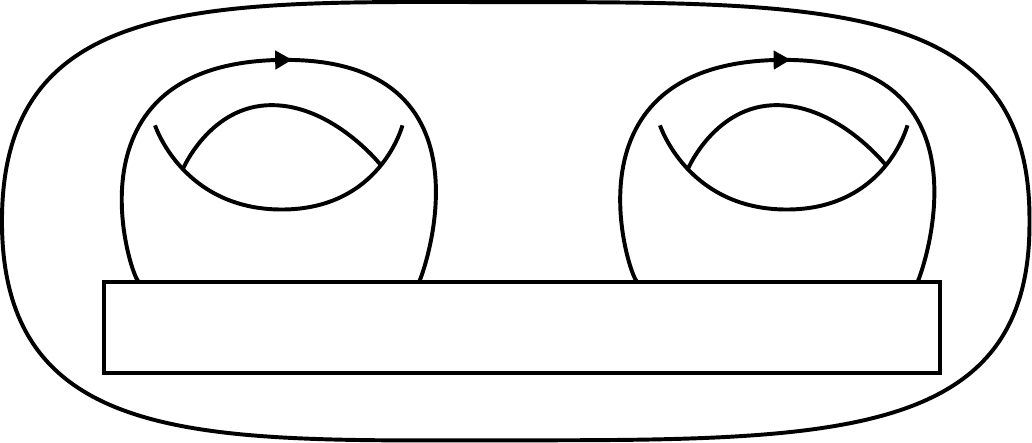}
        \put (50,9.5) {$f_i$}
        \put (10,33) {$i_1$}
        \put (58,33) {$i_g$}
        \put (48,20) {$\cdots$}
    \end{overpic}
    \caption{A handlebody with embedded ribbon graph. The coupon is labelled by a morphism $f_i :\unit \to i_1\otimes i_1^\ast \otimes \dots \otimes i_g\otimes i_g^\ast$.}
    \label{fig:handlebodycoupon}
\end{figure}
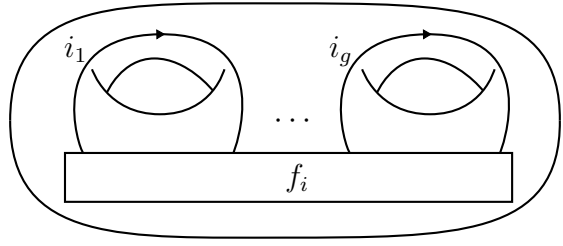

\begin{remark}
\label{rem:RT}
The RT-TQFT is a symmetric monoidal functor from the category of surfaces with $\catname C$-coloured marked points and three-dimensional bordisms with embedded $\catname C$-coloured ribbon graphs (equipped with certain additional decorations) to the category of $\mathbbm{k}$-vector spaces~\cite[Ch.\,IV]{Tur}.
\\
To a surface $\Sigma_g$ without marked points, the RT-TQFT assigns the vector space $V^{\catname C}(\Sigma_g)$. In terms of the TQFT, a vector $f_i \in V^{\catname C}(\Sigma_g)$ is obtained by applying the TQFT-functor to the handlebody shown in 
 Figure~\ref{fig:handlebodycoupon}, thought of as a bordism $\emptyset \to \Sigma_g$.
\\
The action of $[\phi] \in \Mod_g$ is obtained by evaluating the TQFT-functor on the mapping cylinder $\Sigma_g \times [0,1]_\phi$, where the index $\phi$ indicates that one of the boundary parametrisations is given by $\phi$, not by $\id{}$. The projective nature of the representation originates from the extra decorations whose description we skipped.
\\
For example, for the choice of handlebody in Figure~\ref{fig:handlebodycoupon},
the curves $\alpha_1,\dots, \alpha_g$ and $\gamma_1, \dots \gamma_{g-1}$ are contractible and the corresponding Dehn twists will act on the morphism space by twisting the ribbons passing through these curves. This results in the expressions for $T_{\alpha_k}(f_i)$ and $T_{\gamma_k}(f_i)$ given in \eqref{eq:MCG-gen-action}.
\end{remark}

Given two MFCs $\catname C$ and $\catname D$, the Deligne-product $\catname C\boxtimes \catname D$ is obtained by taking pairs of objects, one from $\catname C$ and one from $\catname D$, and tensor products over $\mathbbm{k}$ of Hom-spaces. Finally one completes with respect to direct sums. In particular, for any closed surface $\Sigma$,
\begin{equation}
    V^{\catname C\boxtimes \catname D}(\Sigma) ~\cong~ V^{\catname C}(\Sigma)\otimes_{\mathbbm{k}} V^{\catname D}(\Sigma)
\end{equation}
as $\mathbbm{k}$-vector spaces.
Let us write $\catname C^\mathrm{rev}$ for the category obtained from $\catname C$ by taking the inverse braiding and twist. We will be particularly interested in the product
\begin{equation}
\doublecat := \catname C\boxtimes \catname C^\mathrm{rev}~.
\end{equation}
The version of the object $L$ in \eqref{eq:L-def} for $\doublecat$ will be denoted by $\mathbb{L}$~:
\begin{equation}\label{eq:LL-def}
\mathbb{L} ~:=~ \bigoplus_{i,j \in I}\,  (i \times j) \otimes (i^* \times j^*) ~~\in\,\doublecat~.
\end{equation}
Accordingly, the mapping class group $\Mod_g$ acts on $V^{\doublecat}(\Sigma_g) = \doublecat(\unit\times \unit, \mathbb{L}^{\otimes g})$.

The next lemma follows directly from \cite[Lem.\,VII.4.3.1]{Tur} (see \cite[Sec.\,I.1.4]{Tur} for the definition of $\overline{\catname C}$ used in that lemma, which agrees with our $\catname C^\mathrm{rev}$).

\begin{lemma}\label{lem:doubletheory}
There is an isomorphism $V^{\doublecat}(\Sigma) \cong V^{\catname C}(\Sigma) \otimes_{\mathbbm{k}} V^{\catname C}(\Sigma)^\ast$ which is equivariant with respect to the mapping class group action.  
\end{lemma}

The action of $\Mod_g$ on  $V^{\catname C}(\Sigma_g) \otimes_{\mathbbm{k}} V^{\catname C}(\Sigma_g)^\ast \cong \End_{\mathbbm{k}}(V^{\catname C}(\Sigma_g))$ is by conjugation, thus the projective factors cancel and we obtain a non-projective action.
From Lemma~\ref{lem:doubletheory} we get:

\begin{corollary}\label{cor:mcginvariants}
Let $\Sigma$ be a surface such that $V^{\catname C}(\Sigma)$ is an irreducible projective mapping class group representation. Then the space of mapping class group invariants in $V^{\doublecat}(\Sigma)$ is one-dimensional,
\begin{equation*}
    \dim V^{\doublecat}(\Sigma)^{\Mod(\Sigma)} ~=~ 1~.
\end{equation*}
\end{corollary}

\begin{proof}
The space of mapping class group invariants in $V^{\catname C}(\Sigma) \otimes V^{\catname C}(\Sigma)^\ast$ corresponds to the space of $\Mod(\Sigma)$-equivariant maps $\End_{\Mod(\Sigma)}(V^{\catname C}(\Sigma))$. By Schur's Lemma (for projective representations), the latter is a one-dimensional. Lemma \ref{lem:doubletheory} then implies that the space of mapping class group invariants in $V^{\doublecat}(\Sigma)$ is one-dimensional. 
\end{proof}

\section{Modular invariant Frobenius algebras}
\label{sec:invariants}

In this section, we describe how to obtain mapping class group invariants from a modular invariant symmetric Frobenius algebra. We begin by recalling some algebraic notions following the conventions in \cite{FRS1}. We describe all notions in the MFC $\catname C$, even though they make sense in much greater generality, see e.g.\ \cite{FRS1,FS09} or \cite[Sec.\,7.8]{EGNO} for presentations in more general settings.

\medskip

An \textit{algebra} in $\catname C$ is an object $A \in \catname C$ equipped with morphisms $\eta: \unit \rightarrow A$ (unit) and $\mu: A\otimes A \rightarrow A$ (product), subject to unitality and associativity. It is called \textit{commutative} if $\mu \circ c_{A,A} = \mu$, where $c_{A,A}$ denotes the braiding in $\catname C$.
Dually, a \textit{coalgebra} $C$ is an object with morphisms $\varepsilon: C \rightarrow \unit$ (counit) and $\Delta: C \rightarrow C\otimes C$ (coproduct), which are subject to counitality and coassociativity.

An algebra $A \in \catname C$ is called \textit{simple} if it is simple as a bimodule over itself. It is called \textit{haploid} if $\catname C(\unit,A) = \mathbbm{k}\,\eta$. A haploid algebra is automatically simple \cite[Lem.\,4.5]{Fuchs:2001qc}.

A \textit{Frobenius algebra} is an object $A \in \catname C$ equipped with an algebra and a coalgebra structure satisfying the Frobenius property, which in string diagram notation reads
\begin{equation}\label{eq:Frobprop}
    \boxpic{0.7}{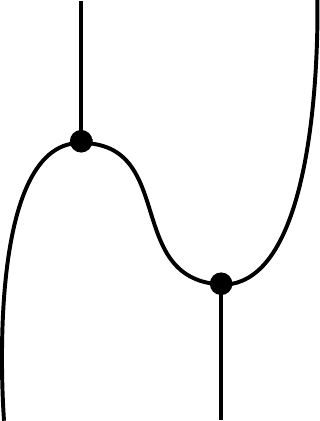}{\put (23,40) {$A$}
    \put (7,70) {$\mu$}
    \put (55,20) {$\Delta$}}~=~\boxpic{0.7}{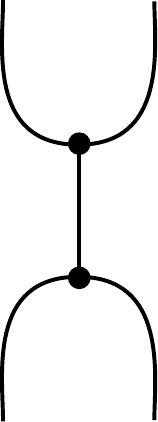}{\put (6,40) {$A$}}~=~\boxpic{0.7}{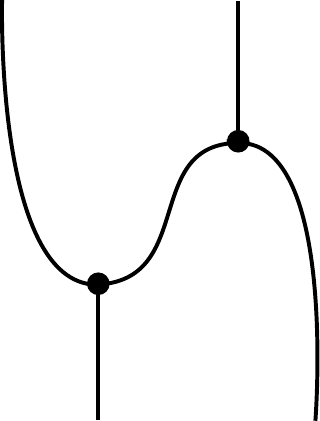}{\put (40,40) {$A$}}
    \quad .
\end{equation}
In this diagram, all lines are labelled $A$ and all diagrams give endomorphisms of $A \otimes A$. The dots label product or coproduct as appropriate for the number of in/out going strands. We will omit the labels for (co)products below.

A Frobenius algebra $A$ 
is \textit{$\Delta$-separable} if $\mu \circ \Delta = \id{}$, and it is 
is \textit{special} if $\varepsilon\circ \eta\neq 0$ and $\mu \circ \Delta = \zeta \id{}$ for some $\zeta \in \mathbbm{k}^\times$. We call $A$ \textit{normalised-special} if $\zeta=1$, or, equivalently, if it is $\Delta$-separable and special.

A Frobenius algebra $ A $ is called \textit{symmetric} if 
\begin{equation}\label{eq:symmetric}
    \boxpic{0.7}{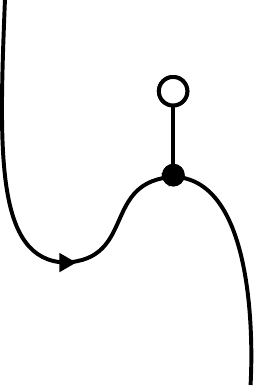}{\put (58,-14) {$A$}
    \put (-5,102) {$A^\ast$}
    \put (50,80) {$\varepsilon$}}~=~\boxpic{0.7}{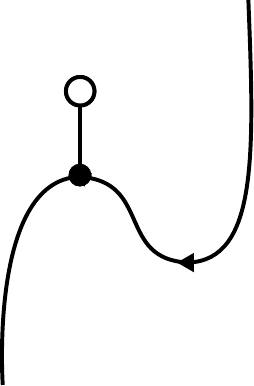}{\put (-5,-14) {$A$}
    \put (60,102) {$A^\ast$}
    \put (25,80) {$\varepsilon$}}
    \quad  .
\end{equation}
As for product and coproduct, we will suppress the label $\varepsilon$ in the string diagram notation of the counit. The notation for the unit $\eta$ is a horizontally flipped version of that for the counit.

Given an algebra $A$, define the morphism $\Phi: A \rightarrow A^\ast$ as
\begin{equation}\label{eq:algebraphi}
    \Phi = \boxpic{0.7}{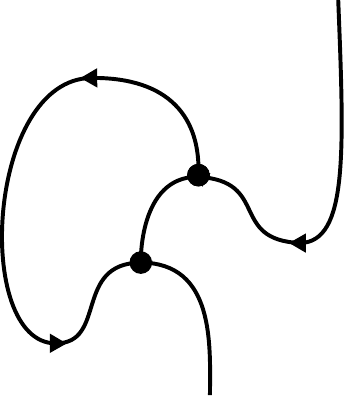}{
    \put (55,0) {$A$}}
    \quad .
\end{equation}
An algebra $A$ is called \textit{non-degenerate} if $\Phi$ is an invertible morphism. 

\begin{lemma}\label{lem:nondeg-ssFa}
Let $A$ be a non-degenerate algebra in $\catname C$. Then:
\begin{enumerate}
    \item There is a unique coproduct and counit on $A$ such that $A$ becomes a symmetric Frobenius algebra on $A$, and such that the isomorphism $A \to A^*$ in \eqref{eq:symmetric} agrees with $\Phi$
    in \eqref{eq:algebraphi}.
\item The Frobenius algebra in part 1 is $\Delta$-separable, and it is normalised-special iff $\dim_{\catname C}(A)\neq 0$.
\item If $A$ is simple, then $\dim_{\catname C}(A)\neq 0$.
\end{enumerate} 
\end{lemma}

Parts 1 and 2 of this lemma are proved in \cite[Lem.\,2.3]{KR}, for part 3 we refer to Appendix~\ref{app:simple-lemma}.

\begin{remark}\label{rem:non-deg-vs-Frob}
The reason why we work with non-degenerate algebras instead of directly with Frobenius algebras is that being non-degenerate is a \textsl{property} of an algebra. Being Frobenius is, first of all,  \textsl{more data} (coproduct and counit). It becomes a property when one adds the conditions of symmetry and $\Delta$-separability. Namely, an algebra is non-degenerate iff it is $\Delta$-separable symmetric Frobenius, cf.\ \cite[Lem.\,2.3]{KR}
\end{remark}

For $B \in \doublecat$ let $\{ \alpha \}$ be a basis of $\doublecat(i\times j,B)$
and let $\{ \overline\alpha \}$ be the dual basis of $\doublecat(B, i\times j)$ in the sense that
$\overline\alpha \circ \beta = \delta_{\alpha,\beta} \id{i \times j}$.
A key ingredient in our proof will be the notion of a modular invariant algebra from \cite[Def.\,3.1]{KR1} (using the alternative formulation in \cite[Lem.\,3.2]{KR1}).

\begin{definition}\label{def:modinv-alg}
An algebra $B$ in $\doublecat$ is called modular invariant if $\theta_B = \id B$ and if the product is \textit{$S$-invariant}, \ie
\begin{equation*}
    \boxpic{0.7}{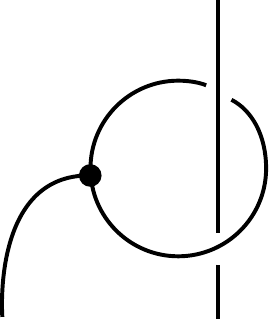}{
        \put (-10,-16) {$B$}
        \put (50,102) {$i\times j$}
        \put (50,-16) {$i\times j$}
        }
~=~
\frac{D^2}{d_i d_j}
\sum_{\alpha}
    \boxpic{0.7}{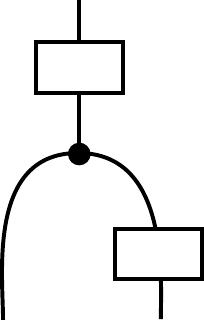}{
        \put (20,73.5) {\small$\overline\alpha$}
        \put (43,16.5) {\small$\alpha$}
        \put (-10,-16) {$B$}
        \put (7,102) {$i\times j$}
        \put (32,-16) {$i\times j$}}
        \rule[-3.2em]{0em}{1em}
        \quad .
\end{equation*}
\end{definition}
Let $B$ be a symmetric Frobenius algebra. For an integer $n\ge 2$ we write $\Delta^{(n)}: B \rightarrow B^{\otimes n}$ for the iterated coproduct, so that $\Delta = \Delta^{(2)}$. For $g \ge 1$
define the elements 
\begin{equation}
C(B)_g ~\in~ V^{\doublecat}(\Sigma_g) = \doublecat(\unit\times\unit, \mathbb{L}^{\otimes g})
\end{equation}
by setting
\begin{equation}\label{eq:definvariants}
    C(B)_g ~:=~ 
        \bigoplus_{i_1,j_1,\dots,i_g,j_g}
        \sum_{\alpha_1,\dots,\alpha_g}
        \boxpic{0.7}{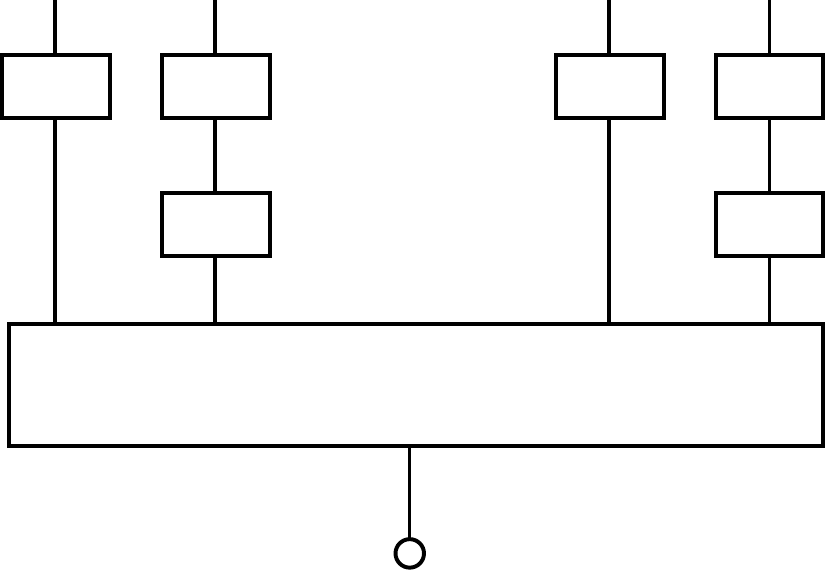}{
            \put (46,20) {$\Delta^{(2g)}$}
            \put (3.5,57) {\small$\overline\alpha_1$}
            \put (22,57.2) {\small$\alpha_1^\ast$}
            \put (71,57) {\small$\overline\alpha_g$}
            \put (90,57.2) {\small$\alpha_g^\ast$}
            \put (23.5,39.3) {\small$\Phi$}
            \put (91,39.3) {\small$\Phi$}
            \put (-3,71) {\small$i_1\times j_1$}
            \put (18,71) {\small$i_1^\ast\times j_1^\ast$}
            \put (64,71) {\small$i_g\times j_g$}
            \put (87,71) {\small$i_g^\ast\times j_g^\ast$}
            \put (52,0) {$\eta$}}
\quad .
\end{equation}
For $g=0$ we have $\Sigma_g = S^2$ and we set
\begin{equation}
    C(B)_0 ~:=~ \varepsilon \circ \eta ~~\in~ 
    V^{\doublecat}(S^2) = \doublecat(\unit\times\unit, \unit\times\unit) = \mathbbm{k} \id{\unit\times\unit} ~.
\end{equation}    

The next proposition follows from \cite{FRS1, KR1, KLR}, but it can also be shown directly  and we give a short proof here for the convenience of the reader.

\begin{proposition}\label{prop:CB-mod-inv}
Let $B \in \doublecat$ be a non-degenerate modular invariant algebra. Then for each 
$g \ge 0$ the vector $C(B)_g \in V^{\doublecat}(\Sigma_g)$ is $\Mod_g$-invariant.
\end{proposition}
\begin{proof}
For $g=0$ there is nothing to show. Let thus $g\ge 1$. We need to check that the generators in \eqref{eq:MCG-gen-action} (for the MFC $\doublecat$) leave $C(B)_g$ invariant.
\begin{itemize}
\item[$T_{\alpha_k}$:]
Invariance is immediate from the fact that $B$ has a trivial twist, \ie $\theta_B = \id{B}$.

\item[$T_{\gamma_k}$:] Choose the iterated coproduct $\Delta^{(2g)}$ such that the $2k$'th and $(2k+1)$'th strand form the output of one coproduct, i.e.\ write
$$
    \Delta^{(2g)} = \big(\id{B^{\otimes (2k-1)}} \otimes \Delta \otimes \id{B^{\otimes(2g-2k-1)}}\big) \circ \Delta^{(2g-1)} ~.
$$
Invariance under $T_{\gamma_k}$ now boils down to the observation that $\theta_{B \otimes B} \circ \Delta = \Delta \circ \theta_B = \Delta$.

\item[$S_k$:] Choose the iterated coproduct $\Delta^{(2g)}$ such that the  $(2k -1)$'th and $2k$'th strand form the output of one coproduct. Applying $S_k$ to $C(B)_g$ only affects the $(2k -1)$'th and $2k$'th strand, and there we obtain:
\begin{align*}
&\frac{d_{i_k}d_{j_k}}{D^2}\sum_{m,n \in I} \sum_{\alpha_k} 
    \boxpic{0.7}{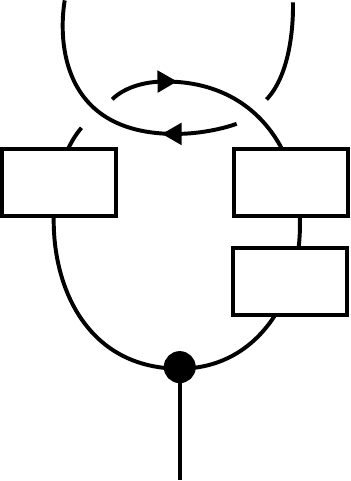}{
        \put (6,59) {\small$\overline\alpha_k$}
        \put (53,59.6) {\small$\alpha_k^\ast$}
        \put (56,37.5) {\small $\Phi$}
        \put (32,-11) {$B$}
        \put (-5,102) {\small$i_k\times j_k$}
        \put (44,102) {\small$i_k^\ast\times j_k^\ast$}
        \put (20,87) {\small $m\times n$}} 
~\overset{(1)}=~ \frac{d_{i_k}d_{j_k}}{D^2}
    \boxpic{0.7}{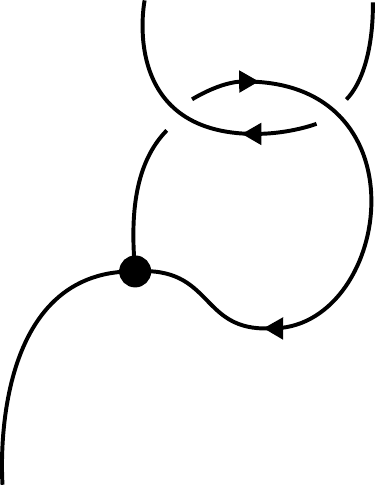}{
        \put (-5,-11) {$B$}
        \put (13,102) {\small$i_k\times j_k$}
        \put (60,102) {\small$i_k^\ast\times j_k^\ast$}}\\[3ex]
&\overset{(2)}=~ 
\sum_{\alpha_k} 
    \boxpic{0.7}{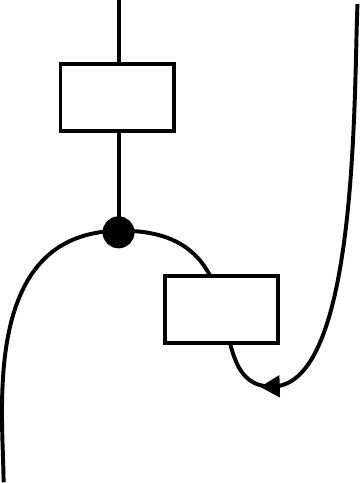}{
            \put (19,77) {\small$\overline\alpha_k$}
            \put (40,34) {\small$\alpha_k$}
            \put (-5,-11) {$B$}
            \put (8,102) {\small$i_k\times j_k$}
            \put (57,102) {\small$i_k^\ast\times j_k^\ast$}} 
~\overset{(3)}=~ 
\sum_{\alpha_k}
        \boxpic{0.7}{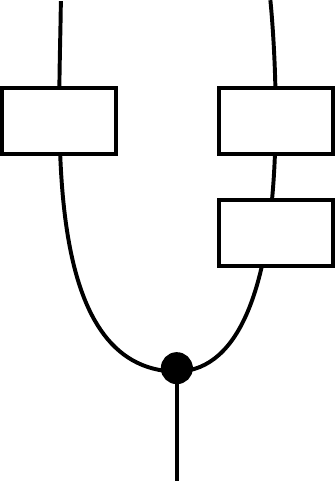}{
            \put (6,72) {\small$\overline\alpha_k$}
            \put (52,72.5) {\small$\alpha_k^\ast$}
            \put (53,47.5) {\small $\Phi$}
            \put (31,-11) {$B$}
            \put (-4,102) {\small$i_k\times j_k$}
            \put (41,102) {\small$i_k^\ast\times j_k^\ast$}}
\end{align*}
For the first expression in this computation, 
recall that we have to evaluate the formula for $S_k$ in \eqref{eq:MCG-gen-action} for $\doublecat$. We take $j \leadsto i_k \times j_k$ and $i_k \leadsto m \times n$ in \eqref{eq:MCG-gen-action}, so that the prefactor there becomes $d_{i_k} d_{j_k}/D^2$. 
In step (1) we carry out the sum over $m,n$ and $\alpha_k$ which gives the identity on $B$, and we use that $B$ is $\Delta$-separable and symmetric to remove $\Phi$.
Step (2) is precisely $S$-invariance of $B$ as in Definition~\ref{def:modinv-alg}.
Step (3) is easier to see backwards, and again uses that $B$ is $\Delta$-separable and symmetric.

This shows that $S_k \circ C(B)_g = C(B)_g$.
\end{itemize}
\end{proof}

\noindent
\begin{minipage}\textwidth
\begin{remark}~
\begin{enumerate}
    \item 
The construction of mapping class group invariants as in \eqref{eq:definvariants} first appeared in the study of consistent systems of correlators for rational 2d conformal field theories via 3d topological quantum field theories \cite{FRS1,Fjelstad:2005ua}. There, $V^{\doublecat}(\Sigma)$ describes the space of holomorphic times antiholomorphic conformal blocks, and a vector $\mathrm{Cor}(\Sigma) \in V^{\doublecat}(\Sigma)$ describes a bulk correlation function on $\Sigma$. To be consistent, the collection $\{ \mathrm{Cor}(\Sigma) \}$ has to satisfy modular invariance and factorisation conditions. Here, we only make use of the former.

\item
The categorical form of the modular invariance condition for algebras first appeared in \cite[Sec.\,6.1]{Kong:2009fp} in the context of vertex operator algebras and has been investigated in detail in \cite{KR1}. The notion of a Cardy algebra from \cite{Kong:2009fp}  was used in \cite{KLR} to classify solutions to the open/closed factorisation and modular invariance conditions. In this context, the algebra $B$ in Proposition~\ref{prop:CB-mod-inv} corresponds to the closed part of a Cardy algebra, and \eqref{eq:definvariants} is the correlator for a closed genus-$g$ surface.

\item
The classification of solutions to the consistency conditions in \cite{KLR} relied on semisimplicty of $\catname C$ and $\doublecat$. A more general approach applicable to non-semisimple modular tensor categories has been developed in \cite{Fuchs:2013lda,Fuchs:2016wjr}. See in particular \cite[Eqn.\,(5.3)]{Fuchs:2013lda} for the generalisation of \eqref{eq:definvariants} and \cite[Def.\,4.9]{Fuchs:2016wjr} for the definition of modular invariant algebras in this non-semisimple setting. These ingredients will be important when trying to generalise the present results to non-semisimple modular tensor categories.
\end{enumerate}
\end{remark}
\end{minipage}

\section{The full centre}\label{sec:full-centre}

In this section, we recall the definition of the full centre of an algebra, as well as a result from \cite{KR} that will be used later for the proof of our main theorem. 

\begin{definition}
The \textit{left centre} $C_l(A)$ of a non-degenerate algebra $A$ is the image of the idempotent $P_l: A \rightarrow A$,
    \begin{equation*}
       P_l~=~ \boxpic{0.7}{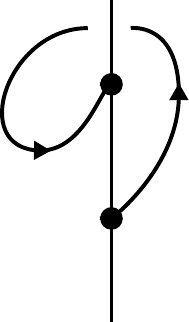}{
            \put (25,-16) {$A$}
            \put (25,104) {$A$}}\quad .
    \end{equation*}
\end{definition}

More details on the definition of left (and right) centres and their properties can be found e.g.\ in \cite[Sec.\,2.4]{Frohlich:2003hm}.

The tensor functor $T: \doublecat \rightarrow \catname C$, $X\times Y \mapsto X\otimes Y$ admits a two-sided adjoint. Explicitly, the adjoint is given by $R: \catname C\rightarrow \doublecat$, $X\mapsto \bigoplus_{i \in I} (X\otimes i^\ast) \times i$, see \cite[Sec.\,2.4]{KR1}.
\begin{definition}\label{def:fullcenter}
Let $A \in \catname C$ be a non-degenerate algebra. The \textit{full centre of $A$} is  $Z(A) = C_l(R(A)) \in \doublecat$. 
\end{definition} 

\begin{remark}
The full centre was first introduced in \cite[Def.\,4.9]{Fjelstad:2006aw}. Actually, one can assign to an algebra $A$ in a monoidal category $\catname M$ a commutative algebra in the Drinfeld centre $\catname{Z(M)}$ which is characterised by a universal property \cite{Davydov:2009}. The notion in Definition~\ref{def:fullcenter} is a special case of this more general characterisation.
\end{remark}

The full centre is important in our construction because it produces modular invariant algebras. The following theorem is the first key input in our construction. It is shown in \cite[Prop.\,2.7]{KR} and \cite[Thm.\,3.18]{KR1}.

\begin{theorem}\label{thm:fullcentre}
Let $A \in \catname C$ be a simple non-degenerate algebra. Then the full centre $Z(A) \in \doublecat$ is a haploid commutative non-degenerate modular invariant algebra with $\dim_{\doublecat} Z(A) = D^2$.
\end{theorem}

\begin{example}\label{ex:Z(1)-const}
The fundamental example is to choose $A=\unit \in \catname C$.
We describe the Frobenius algebra structure of $Z(\unit)$ explicitly, as we will need it later.
The expressions below are taken from \cite[Eq.\,(2.58)]{KR1}, which gives $R(A)$, together with the observation that for $A = \unit$ it is already commutative, and so equal to $Z(\unit)$.
The underlying object of $Z(\unit)$ is $\bigoplus_{i\in I} i^\ast\times i$. The unit is given by the natural embedding of $\unit \times \unit$, while the counit is given by the projection
to $\unit \times \unit$ times $D^2$. Let $\{\alpha\}$ be a basis of $\catname C(i \otimes j, k)$ and $\{\overline\alpha\}$ the dual basis in $\catname C(k,i \otimes j)$ in the sense that $\alpha \circ \overline\beta = \delta_{\alpha,\beta}$.
The product and coproduct are given by  
\begin{align}
\mu_{Z(\unit)}~&=~ \bigoplus_{i,j,k} \sum_{\alpha= 1}^{N_{ij}^{~k}} \boxpic{0.7}{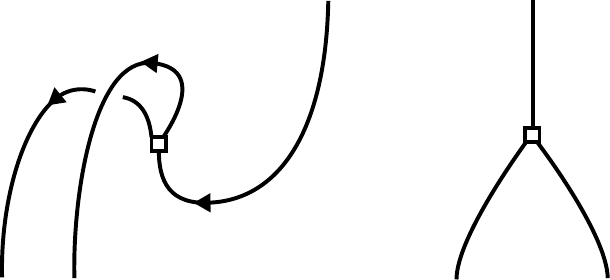}{
			 		\put (12, -5) {$ j^\ast $}
			 		\put (2, -5) {$ i^\ast $}
			 		\put (51, 47) {$ k^\ast $}
			 		\put (76, -5) {$ i $}
			 		\put (100, -5) {$ j $}
			 		\put (85, 47) {$ k $}
			 		\put (30, 20) {$\overline \alpha $}
			 		\put (92, 22) {$\alpha$}
			 		\put (60, 15) {$\otimes_\mathbbm{k}$}}
\rule[-3.2em]{0em}{7em}			 		
\quad ,
\nonumber\\
\Delta_{Z(\unit)}~&=~ \bigoplus_{i,j,k} \sum_{\alpha= 1}^{N_{ij}^{~k}} \frac{d_i d_j}{d_k \,D^2}\boxpic{0.7}{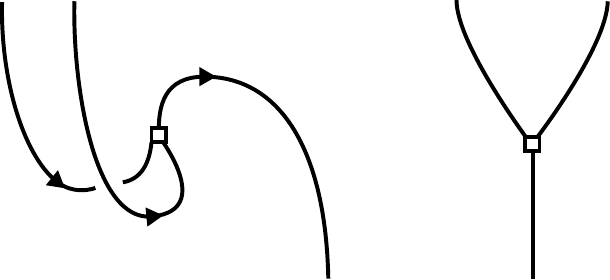}{
			 		\put (14, 47) {$ j^\ast $}
			 		\put (0, 47) {$ i^\ast $}
			 		\put (54.5, -5) {$ k^\ast $}
			 		\put (75, 47) {$ i $}
			 		\put (101, 47) {$ j $}
			 		\put (88, -5) {$ k $}
			 		\put (30, 21) {$\alpha $}
			 		\put (92, 20) {$\overline \alpha $}
			 		\put (60, 15) {$\otimes_\mathbbm{k}$}}
\rule[-3em]{0em}{6.5em}			 		
			 		\quad .
\label{eq:Z1-product-coproduct}
\end{align}
\end{example}

We briefly recall the notions of bimodules and Morita equivalence. 
Let $ A$ and $B$ be two algebras. An $A$-$B$-\textit{bimodule} $ T $ carries a left $ A $-action $ A\otimes T\rightarrow T $ 
and a right $ B $-action $ T\otimes B\rightarrow T$ 
which commute with each other. Two algebras $A$ and $B$ are called \textit{Morita equivalent}, if there exist an $A\mbox{-}B$-bimodule $X$ and a $B\mbox{-}A$-bimodule $Y$ such that $X\otimes_B Y \cong A$ and $Y\otimes_A X \cong B$ as bimodules. 
			
The next theorem is the second key input for our construction, as it relates Morita equivalence to isomorphisms of full centres.

\bigskip

\noindent
\begin{minipage}{\textwidth}
\begin{theorem}[{\cite[Thm.\,1.1]{KR}}]\label{thm:moritaiso}
Let $A$ and $B$ be simple non-degenerate algebras. Then the following are equivalent: 
\begin{enumerate}
    \item $A$ and $B$ are Morita equivalent.
    \item $Z(A)$ and $Z(B)$ are isomorphic as algebras.
\end{enumerate}
\end{theorem}
\end{minipage}

\section{Main theorem}\label{sec:mainthm}

We have now gathered the ingredients we need to state and prove our main theorem:

\begin{theorem}
\label{thm:maintheorem}
Let $\catname C$ be a MFC such that the projective mapping class group representations $V_g^{\catname C}$ are irreducible for all $g \ge 0$. Then $\catname C$ has a unique Morita class of simple non-degenerate algebras, namely the Morita class of the tensor unit $\unit$.
\end{theorem}

The proof is contained in Sections~\ref{sec:structureconstants} and \ref{sec:isosequence}.

\medskip

In \cite[Thm.\,1]{Andersen:2008} the following closely related statement is shown: 
\begin{quote}
Let $A \in \catname C$ be a simple non-degenerate algebra such that $Z(A)$ is not isomorphic to $Z(\unit)$ as an object in $\doublecat$. Then all projective mapping class group representations $V_g^{\catname C}$, $g \ge 1$ are reducible.
\end{quote}
In contrapositive form this reads: Suppose there is a $g \ge 1$ such that $V_g^{\catname C}$ is irreducible. Then for every simple non-degenerate algebra $A$ one has that $Z(A)$ is isomorphic to $Z(\unit)$ as an object in $\doublecat$.

From this point of view, on the one hand, Theorem~\ref{thm:maintheorem} needs the stronger assumption that $V_g^{\catname C}$ is irreducible for all $g \ge 0$ (however, see Remark~\ref{rem:maintheorem}\,(1) below). On the other hand, under these assumptions it gives a stronger result, namely together with Theorem~\ref{thm:moritaiso} it follows that $Z(A) \cong Z(\unit)$ as algebras in $\doublecat$, and not just as objects. This confirms an expectation formulated in \cite[Rem.\,1]{Andersen:2008}, at least under our stronger assumptions. Our method to prove Theorem~\ref{thm:maintheorem} is different from that used in \cite{Andersen:2008}, and thus may be of independent interest.

We note that it is not at all obvious that there are examples where $Z(A) \cong Z(\unit)$ as objects but not as algebras. Such examples were first provided in \cite{Davydov:2014}.\footnote{In these examples it is not required that all simple non-degenerate algebras have $Z(A) \cong Z(\unit)$ as objects. Thus these examples do not yet imply that the conclusion of Theorem~\ref{thm:maintheorem} is indeed stronger than that of \cite[Thm.\,1]{Andersen:2008}.} In fact, that paper provides examples of Lagrangian algebras, but each such algebra can be realised as a full centre by \cite[Thm.\,3.22]{KR1} (see also \cite[Prop.\,4.8]{DMNO} for a more general statement).

\begin{remark}\label{rem:maintheorem}~
\begin{enumerate}
    \item In the proof of Theorem~\ref{thm:maintheorem} we actually need irreducibility of the representations $V_g^{\catname C}$ only for $1 \le g\le 3N+2$, where $N$ is a bound introduced in Section~\ref{sec:universalgroup} in terms of the adjoint subring. 
    The place in the proof where this maximal $g$ occurs is pointed out in Remark~\ref{rem:maximal-g}.
    The constant $N$ in turn is trivially bounded by the number of isomorphism classes of simple objects, $N\le |I|$. In other words, one can relax the hypothesis of Theorem~\ref{thm:maintheorem} to assume irreducibility only for $V_g^{\catname C}$ with $1 \le g \le 3N+2$.
    
    \item 
In this paper we exclude surfaces with marked points. Nonetheless, let us for the moment consider the surface $\Sigma_{0,3}$, \ie the sphere with three punctures, and assume that the punctures are labelled by simple objects, say $i,j,k \in I$. The (framed, pure) 
mapping class group $\Mod(\Sigma_{0,3})$ acts on $V^{\catname C}(\Sigma_{0,3})$ by rotation of the framing at the marked points, and so by a scalar given by the corresponding twist eigenvalue. If $V^{\catname C}(\Sigma_{0,3})$ is non-zero, for $\Mod(\Sigma_{0,3})$ to act irreducibly we must hence have $\dim V^{\catname C}(\Sigma_{0,3})=1$.

On the other hand, $V^{\catname C}(\Sigma_{0,3}) = \catname C(\unit, i \otimes j \otimes k)$. Thus, requiring irreducibility of the mapping class group action also on surfaces with marked points implies in particular that the fusion coefficients of $\catname C$ must satisfy $N_{ij}^{~\kbar} \in \{0,1\}$. Considering only surfaces without punctures, as we do, does not a priori impose this requirement, but we do not know any example where the $V^{\catname C}_g$, $g\ge 0$ are irreducible but $N_{ij}^{~\kbar}>1$ can occur.
\end{enumerate}
\end{remark}

\begin{example} \label{ex:irred-or-not}
The only examples with irreducible $V^{\catname C}_g$'s we are aware of are Ising-type categories and the MFC $\catname C(sl(2),k)$ associated to the affine Lie algebra $\widehat{sl}(2)$ at certain levels $k \in \mathbb{Z}_{>0}$.
Let us list these examples, as well as some non-examples. (In all these examples it was already known that there is a unique Morita-class of simple non-degenerate algebras.)
\begin{enumerate}
\item 
It is shown in \cite{Roberts:1999} that for $\catname C = \catname C(sl(2),k)$ and $r = k+2$ prime,
all projective representations $V^{\catname C}_g$, $g \ge 0$ are irreducible.

Most of the remaining cases can be excluded already by looking at $g=1$. Namely by \cite[App.\,A]{Gepner:1986hr} and \cite[Prop.\,1]{Cappelli:1987xt}, invariants in the representation $V^{\doublecat}_{g=1}$ are obtained from divisors $d$ of $r$, with divisors $d$ and $r/d$ describing the same invariant subspace, and where $d$ with $d^2=r$ is excluded.
Thus, when $r \ge 3$ is not a prime or a square of a prime,
the space of invariants satisfies
$\dim (V^{\doublecat}_{1})^{\Mod_{1}}>1$, and so by Corollary~\ref{cor:mcginvariants}, $V^{\catname C}_{1}$ is not irreducible. 

    On the other hand, for $k=2$ ($r = 4$), one obtains a category of Ising-type, for which all $V^{\catname C}_{g}$ are irreducible, see point 2. 
    Some results on the irreducibility of $V^{\catname C}_{g \ge 2}$ for the remaining cases of $r = p^2$ with $p>2$ prime can be found in \cite{Korinman}.

\item 
The Ising model without marked points is studied in \cite{Castro:2011zq,Jian:2019ubz}. Irreducibility of all $V^{\catname C}_g$, $g \ge 0$ is shown in \cite[Sec.\,4.3]{Jian:2019ubz}. In \cite{RR-prep} we will extend this result to all 16 Ising-type MFCs. 

\item 
Let $\catname C = \catname C(sl(N),k)$ be the MFC for the affine Lie algebra $\widehat{sl}(N)$ at any
level $k \in \mathbb{Z}_{>0}$, for $N \ge 3$. It is shown in \cite[Thm.\,3.6]{Andersen:2009} that the $V^{\catname C}_g$ are reducible for each $g \ge 1$. 

\item 
For $\catname C(sl(2),k)$, irreducibility has also been studied for the mapping class group of surfaces with marked points, see~\cite{KoSa,KuMi}. For Ising-type MFCs, irreducibility in the presence of marked points will be shown in \cite{RR-prep}.
\end{enumerate}
\end{example}

Theorem~\ref{thm:maintheorem} can be reformulated using module categories. Namely, a \textit{$\catname C$-module category} is a category $\catname M$ together with a functor $\catname C \times \catname M \to \catname M$ and coherence isomorphisms, subject to associativity and unit conditions. A module category is \textit{indecomposable} if it is not equivalent, as a $\catname C$-module category, to a direct sum of non-trivial module categories. 
It is shown in \cite[Sec.\,3.3]{Os} that there is a one-to-one correspondence between Morita-classes of simple algebras in $\catname C$ and semisimple indecomposable $\catname C$-module categories (given by passing from an algebra $A$ to the category of right $A$-modules in $\catname C$). 

In order to have a correspondence with Frobenius algebras, one needs to equip the module categories with a \textit{module trace} \cite{Schau}, \ie a family of maps $\{\Theta_M\}_{M \in \catname M}$ with $\Theta_M:\End(M) \to \mathbbm{k}$, subject compatibility conditions with the pivotal structure of $\catname C$, see \cite[Sec.\,3.2]{Schau}.
From \cite[Thm.\,6.6, Prop.\,6.8]{Schau} we get the following reformulation of Theorem~\ref{thm:maintheorem}:

\newtheorem{otherthm}{Theorem}
\numberwithin{otherthm}{section}

\begin{otherthm}[v2]
\label{thm:otherversion}
Let $\catname C$ be a MFC such that the projective mapping class group representations $V_g^{\catname C}$ are irreducible for all $g \ge 0$. Then there is up to equivalence a unique semisimple indecomposable $\catname C$-module category with module trace, namely $\catname C$ itself.
\end{otherthm}

As an application of this point of view, let us explain how under certain conditions the non-degeneracy of a simple algebra is implied. The MFC $\catname C$ is called \textit{pseudo-unitary} if $\mathbbm{k}=\mathbb{C}$ and if the quantum dimensions of all simple objects are positive. 
By \cite[Prop.\,5.8]{Schau}, for pseudo-unitary $\catname C$, a semisimple $\catname C$-module category can be equipped with a module trace. Hence in this situation we can drop the existence of a module trace from Theorem~\ref{thm:otherversion}\,(v2).
We obtain the following corollary to Theorem~\ref{thm:maintheorem} (see also \cite[Cor.\,6.11]{Schau}):

\begin{corollary}\label{cor:main-pseudounitary}
Suppose that in addition to the hypotheses in Theorem~\ref{thm:maintheorem}, $\catname C$ is pseudo-unitary. Then all simple algebras in $\catname C$ are Morita-equivalent to the tensor unit.
\end{corollary}

\medskip

Before going into the details, let us briefly sketch the proof of 
Theorem~\ref{thm:maintheorem}. By Theorem~\ref{thm:moritaiso} it suffices to show that for any simple non-degenerate algebra $A$ we have $Z(A)\cong Z(\unit)$ as algebras. To obtain this isomorphism we proceed in several steps:
\begin{enumerate}
    \item In Section \ref{sec:universalgroup} we will review the notion of the adjoint subring and universal grading group as well as the bound $N$ mentioned in Remark~\ref{rem:maintheorem}.

    \item In Section \ref{sec:structureconstants}, we will use irreducibility on the torus and obtain multiplication constants $(\lambda^k_{ij})^\alpha_\beta$ relating the structure morphisms of $Z(A)$ to those of $Z(\unit)$. Furthermore, we use irreducibility for genus~2 to obtain constants $\lambda^k_{ij}$ independent of the multiplicity labels $\alpha, \beta$. We then use irreducibility for $g>2$ to obtain constraints on the $\lambda^k_{ij}$.

    \item In Section \ref{sec:isosequence} we construct a sequence of algebra isomorphisms using the results of the previous step and the universal grading group to arrive to an algebra isomorphism $Z(A)\cong Z(\unit)$.
\end{enumerate}

\subsection{The adjoint subring and the universal grading group}\label{sec:universalgroup}

We briefly recall from \cite{Gelaki:2006} the notion of the universal grading group and of the adjoint subring (see also \cite[Ch.\,3]{EGNO}).

Let $\catname{F}$ be a fusion category and $I$ a set of representatives of isomorphism classes of simple objects in $\catname F$.
The duality on $\catname{F}$ defines an involution $\overline{(\;)}: I \rightarrow I$ by requiring that $\ibar\cong i^\ast$. Then, the Grothendieck ring $\mathrm{Gr}(\catname F) \equiv R$ is a unital based ring with basis 
$\{b_i\}_{i\in I}$. The ring $R$ is \textit{transitive} in the sense that for any $i,j \in I$ there exists $k\in I$ such that $N_{ik}^{j}\neq 0$.

\begin{definition}
The \textit{adjoint subring} $R_{ad}\subset R$ is generated by all basis elements contained in $b_ib_{\ibar}$ for some $i \in I$. We denote by $I_{ad}\subset I$ the index set of the basis $\{b_i\}_{i\in I_{ad}}$ of $R_{ad}$.  
\end{definition}

It is shown in \cite[Thm.\,3.5]{Gelaki:2006} that
the ring $R$ decomposes into a direct sum of indecomposable based
$R_{ad }$-bimodules $R= \bigoplus_{g\in G} R_g$, and that
the product of $R$ induces a group structure on $G$ with $R_e= R_{ad}$. 
In particular, $R$ is a faithful $G$-graded ring. The set $I_{g}\subset I$ will denote the index set of the basis $\{b_i\}_{i\in I_{g}}$ of $R_{g}$.
Transitivity of $R$ now implies that $R_{ad}$ acts transitively on $R_g$ for each $g \in G$: for all $x,y \in I_g$ there is $i \in  I_{ad}$ such that $N_{xi}^{~y}\neq 0$.

\begin{definition}
The group $G$ is called the \textit{universal grading group of $R$}. 
\end{definition}

We define a filtration on $R_{ad}$ as follows. For $i\in I_{ad}$ let $n(i)$ be the minimal integer such that $b_i$ is contained in $b_{m_{1}}b_{\overline{m}_{1}}\dots b_{m_{n(i)}}b_{\overline{m}_{n(i)}}$ for some $m_1,\dots,m_{n(i)} \in I$. 
Such labels exist by the definition of the adjoint subring. For the unit we set $n(\unit )=0$.
Setting
\begin{equation}\label{eq:filtrationdef}
R^{(n)}_{ad} ~=~ \langle\, b_i \in R_{ad}~|~n(i)\leq n\,\rangle~.
\end{equation}
we get the filtration 
\begin{equation}\label{eq:filtration}
\mathbb{Z} \, b_{\unit} = R^{(0)}_{ad} \subset R^{(1)}_{ad} \subset R^{(2)}_{ad} \subset \cdots 
\end{equation}
of $R_{ad}$. 
Let $N$ denote the minimal number such that $R_{ad}^{(N)} = R_{ad}$. Since the filtration is strictly increasing until degree $N$, and since $I_{ad} \subset I$, we trivially have that $N \le |I|$.

\subsection{Structure constants}\label{sec:structureconstants}
Let $A \in \catname C$ be a simple non-degenerate algebra. By Theorem \ref{thm:fullcentre} and Lemma \ref{lem:nondeg-ssFa} the full centre $Z(A)$ is a haploid normalised-special commutative symmetric modular invariant Frobenius algebra.

\begin{lemma}\label{lem:ZA-object}
The full centre $Z(A)$ has the same underlying object as $Z(\unit)$, \ie $Z(A)\cong \bigoplus_{i\in I} i^\ast\times i$ as objects in $\doublecat$. 
\end{lemma}

\begin{proof}
By \cite[Eq.~(3.7)]{KR1}, the matrix $Z(A)_{ij} = \dim \doublecat(i\times j, Z(A))$ commutes with the $S$-generator, and it commutes with the $T$-generator since $Z(A)$ has trivial twist. By irreducibility of $V_{g=1}^{\catname C}$, the space of invariants in $V_{g=1}^{\doublecat}$ is one-dimensional (Corollary~\ref{cor:mcginvariants}). Hence
there exists a constant $\lambda \in \mathbbm{k}$ 
such that 
\begin{equation}
Z(A)_{ij} =\lambda\, Z(\unit)_{ij} = \lambda\, \delta_{\ibar j}~.
\end{equation}
By haploidity,
$Z(A)_{\unit \unit} = 1$ and therefore $\lambda= 1$. Altogether, 
$\dim \doublecat(i\times j, Z(A)) = \delta_{\ibar, j}$
\ie the underlying object of $Z(A)$ is $\bigoplus_{i\in I} i^\ast \times i$. 
\end{proof}

Denote by $e_i: i^\ast\times i \rightarrow Z(A)$ and $r_i: Z(A)\rightarrow i^\ast\times i$ the embedding and projection
of $i^\ast\times i$ as a subobject of
$Z(A)$, \ie $r_i \circ e_i = \id{}$. Given the underlying object of $Z(A)$ as in Lemma~\ref{lem:ZA-object}, we now make a general ansatz for the Frobenius algebra structure on $Z(A)$. Namely, in terms of constants $\eta_0$, $\varepsilon_0$, $(\lambda_{ij}^k)^\beta_\alpha$, $(\lambda_{k}^{ij})^\alpha_\beta \in \mathbbm{k}$ we set
\begin{align}
    \eta_{Z(A)} &~=~ \eta_0 \, e_{\unit}
\nonumber\\[.5em]
    \varepsilon_{Z(A)} &~=~ \varepsilon_0\, D^2 \, r_{\unit}
\nonumber\\
\mu_{Z(A)}~&=~ \bigoplus_{i,j,k} \sum_{\alpha,\beta= 1}^{N_{ij}^{~k}} (\lambda_{ij}^k)^\beta_\alpha\boxpic{0.7}{figures/multconstants.pdf}{
			 		\put (12, -5) {$ j^\ast $}
			 		\put (2, -5) {$ i^\ast $}
			 		\put (51, 47) {$ k^\ast $}
			 		\put (76, -5) {$ i $}
			 		\put (100, -5) {$ j $}
			 		\put (85, 47) {$ k $}
			 		\put (30, 20) {$\overline \alpha $}
			 		\put (92, 20) {$\beta $}
			 		\put (60, 15) {$\otimes_\mathbbm{k}$}}
\nonumber\\[1em]
    \Delta_{Z(A)}
~&=~ \bigoplus_{i,j,k} \sum_{\alpha,\beta= 1}^{N_{ij}^{~k}} \frac{d_i d_j}{d_k D^2}(\lambda_{k}^{ij})^\alpha_\beta\boxpic{0.7}{figures/comultconstants.pdf}{
		\put (14, 47) {$ j^\ast $}
		\put (0, 47) {$ i^\ast $}
		\put (54.5, -5) {$ k^\ast $}
		\put (75, 47) {$ i $}
		\put (101, 47) {$ j $}
		\put (88, -5) {$ k $}
		\put (30, 21) {$\alpha $}
		\put (92, 18) {$\overline \beta$}
		\put (60, 15) {$\otimes_\mathbbm{k}$}}\quad .
\label{eq:ZA-constants}
\end{align}
On the right hand side of $\mu_{Z(A)}$ we did not spell out the embedding and projection
morphisms $r_k \circ (\cdots) \circ (e_i \otimes e_j)$, and dito for $\Delta_{Z(A)}$.

\begin{lemma}\label{lem:CZ1-nonzero}
The elements $C(Z(A))_g$ from \eqref{eq:definvariants} are non-zero for every $g \ge 0$.
\end{lemma}

\begin{proof}
The element $C(Z(A))_0$ is non-zero since $\varepsilon_{Z(A)}\circ \eta_{Z(A)} = \dim_{\doublecat} (Z(A)) = D^2$, where the first equality follows from the symmetric normalised-special Frobenius algebra structure. 
Now, let $g\ge 1$ and consider in \eqref{eq:definvariants} the summand of $C(Z(A))_g$ where $i_m = j_m = \unit$ for $m=1,\dots,g$. 
Since $\unit \times \unit$ appears in $Z(A)$ with multiplicity one, there is no sum over multiplicities. Up to factors of $\varepsilon_0 D^2 \neq 0$, the result is the same
as composing all out-going $Z(A)$-factors with the counit $\varepsilon_{Z(A)}$. The overall expression then reduces to $\varepsilon_{Z(A)} \circ \eta_{Z(A)} = D^2$.
\end{proof}

As in \cite[Sec.\,2.2]{FRS1}, for every $i\in I$ fix an isomorphism $\pi_i: i \rightarrow \ibar^\ast$, which exists by definition of the involution $i\mapsto \ibar$. We use these isomorphisms to express the fusion basis in $\catname C (i\otimes \ibar, \unit) $ and its dual in terms of dualities in $\catname C$. Namely, there exist $\phi_i, \tilde{\phi}_i \in \mathbbm{k}^\times$ such that 
\begin{equation}\label{eq:pi}
  \boxpic{0.7}{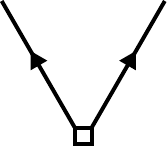}{
  \put (0,95) {$i$}
  \put (90,95) {$\ibar$}
  }
  = \phi_i \boxpic{0.7}{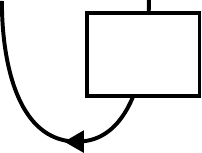}{
  \put (-4,83) {$i$}
  \put (72,83) {$\ibar$}
  \put (52,42.5) {\small$\pi_{\ibar}^{-1}$}
  }
  \quad,\quad \boxpic{0.7}{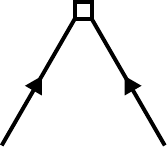}{
  \put (0,-29) {$i$}
  \put (105,-29) {$\ibar$}
  } 
  = \tilde\phi_i \boxpic{0.7}{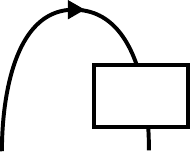}{
  \put (-4,-29) {$i$}
  \put (76,-29) {$\overline i$}
  \put (62,25) {\small$\pi_{\overline i}$}
  }
  \rule[-2.2em]{0em}{4.7em}
\end{equation}
as the respective morphism spaces are one-dimensional. By the normalisation chosen for fusion bases, one obtains 
\begin{equation}\label{eq:phi}
\phi_i\,\tilde{\phi}_i = \frac{1}{d_i}~.
\end{equation}

In the following lemma we will give the isomorphism $\Phi$ for $Z(A)$ in a basis, which will later be used to express the modular invariants $C(Z(A))_g$.

\begin{lemma}\label{lem:phicomputation}
For any $i \in I$, we have 
\begin{equation*}
    e^\ast_i \circ \Phi_{Z(A)} \circ e_{\ibar} ~=~ \frac{D^2 \theta_i}{d_i} \, \varepsilon_0 \,\lambda_{i \ibar}^\unit ~ (\pi_{\ibar}^{-1})^\ast\otimes_\mathbbm{k} \pi_{\ibar}
    \quad :~~ 
    \ibar^\ast\times \ibar \longrightarrow i^{\ast\ast}\times i^\ast
    ~.
\end{equation*}
\end{lemma}

\begin{proof}
By using \eqref{eq:ZA-constants} in the expression for $\Phi$ on the left hand side of \eqref{eq:symmetric},
one obtains: 
\begin{equation*}
    e^\ast_i \circ \Phi_{Z(A)} \circ e_{\ibar} = D^2 \varepsilon_0 \lambda_{i\ibar}^\unit 
    \boxpic{0.7}{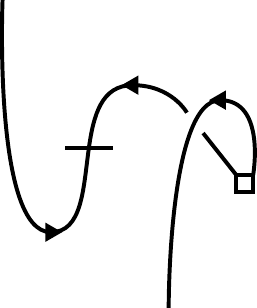}{
    \put (0,105) {$i^{\ast\ast}$}
    \put (50,75) {$i$}
    \put (50,-20) {$\ibar^\ast$}
    }  
    \otimes_\mathbbm{k} \boxpic{0.7}{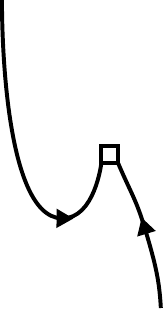}{
    \put (0,105) {$i^{\ast}$}
    \put (50,-20) {$\ibar$}
    }
    \rule[-3.4em]{0em}{6.8em}
    = \frac{D^2 \, \theta_i}{d_i} \, \varepsilon_0 \, 
    \lambda_{i \ibar}^\unit ~ (\pi_{\ibar}^{-1})^\ast \otimes_\mathbbm{k} \pi_{\ibar}
    ~,
\end{equation*}
where the horizontal line denotes the identity $\id{i^\ast}$. The last equality follows from \eqref{eq:pi} and \eqref{eq:phi}. 
\end{proof}

The allowed structure constants are non-zero, diagonal and independent of the multiplicity index:

\begin{lemma}\label{lem:diagonal}
For $i,j,k \in I$ with $N_{ij}^{~k} \neq 0$ we have 
$(\lambda_{ij}^k)_\alpha^\beta = \delta_{\alpha,\beta}\, \lambda_{ij}^k$ and $(\lambda^{ij}_k)_\alpha^\beta = \delta_{\alpha,\beta}\, \lambda^{ij}_k$ with $\lambda_{ij}^k, \lambda^{ij}_k \neq 0$.
\end{lemma}
\begin{proof}
Consider the modular invariant vector $C(Z(A))_{g=2}$ as defined in \eqref{eq:definvariants}, which is given as\footnote{To be precise, the element $C(Z(A))_2$ is defined in the isomorphic morphism space $\bigoplus_{i,k}\doublecat(\unit\times \unit, (\ibar\times i)\otimes (\ibar^\ast\times i^\ast)\otimes (\kbar\times k)\otimes (\kbar^\ast\times k^\ast))$, but we find it convenient to use the form given here, rather than to include the isomorphisms $\pi_{\ibar}^{-1} \times \id{i}$ and $\pi_{\ibar}^* \times \id{i}$, etc.
}
\begin{equation}
    C(Z(A))_2=\bigoplus_{i,k}
    \boxpic{0.7}{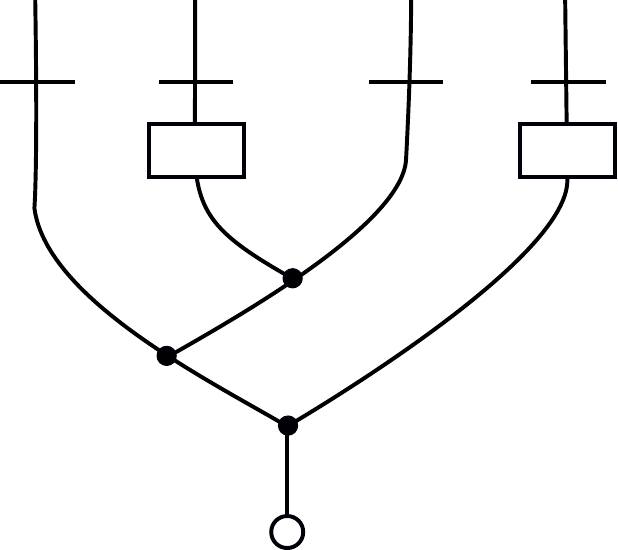}{\put (-8,90) {\small$i^\ast\times i$}
        \put (17,90) {\small$i^{\ast\ast}\times i^\ast$}
        \put (52,90) {\small$k^\ast \times k$}
        \put (80,90) {\small$k^{\ast\ast}\times k^\ast$}
        \put (28,61.5) {\small $\Phi$}
        \put (88,61.5) {\small $\Phi$}
        \put (48,13)  {$Z(A)$}} 
    = \bigoplus_{i,k}
    \boxpic{0.7}{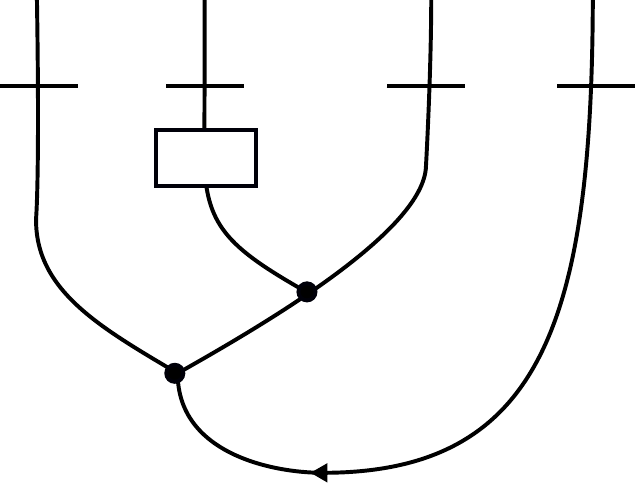}{\put (-7,77) {\small$i^\ast\times i$}
        \put (18,77) {\small$i^{\ast\ast}\times i^\ast$}
        \put (54,77) {\small$k^\ast\times k$}
        \put (83,77) {\small$k^{\ast\ast}\times k^\ast$}
        \put (28.5,48) {\small $\Phi$}
        }
\end{equation}
where the horizontal lines denote the embeddings and projections,
and where
we used that $Z(A)$ is symmetric Frobenius to remove one of the $\Phi$'s. 
Using Lemma~\ref{lem:phicomputation}, we can write $C(Z(A))_2$ explicitly:
\begin{align*}
    \bigoplus_{i,k}\sum_j\sum_{\alpha,\beta,\gamma,\delta}{\frac{d_i\theta_i}{D^2}
    \,
    \varepsilon_0 \, \lambda_{i \ibar}^{\unit} \, (\lambda_{j}^{\ibar k})^\alpha_\beta \, (\lambda_k^{ij})_\delta^\gamma}
    \boxpic{0.7}{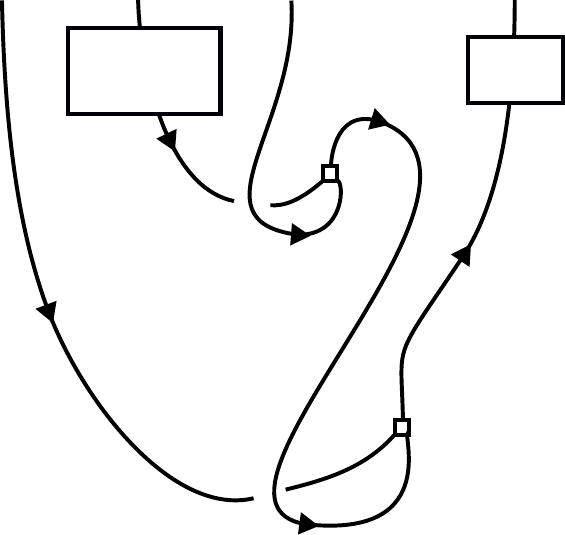}{\put (62,60) {\small$\alpha$}
        \put (75,15) {\small$\gamma$}
        \put (0,97) {\small$i^\ast$}
        \put (23,97) {\small$i^{\ast\ast}$}
        \put (11.7,80) {\small$(\pi_{\ibar}^{-1})^\ast$}
        \put (87,80) {\small$a_k$}
        \put (50,97) {\small$k^\ast$}
        \put (89,97) {\small$k^{\ast\ast}$}
        \put (54,30) {\small $j$}
        }
    \otimes_\mathbbm{k}  
    \boxpic{0.7}{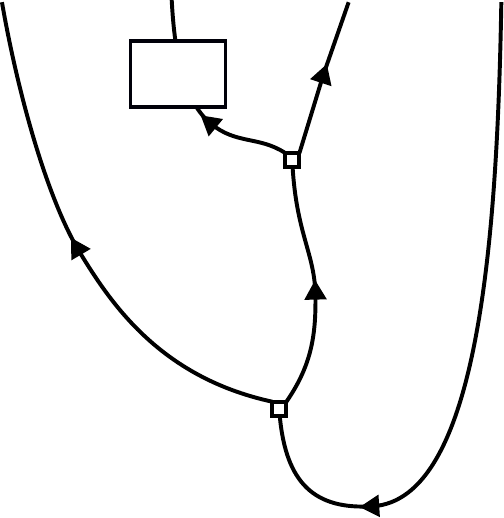}{\put (-3,102) {\small$i$}
        \put (32,102) {\small$i^\ast$}
        \put (30,83.5) {\small$\pi_{\ibar}$}
        \put (65,102) {\small$k$}
        \put (96,102) {\small$k^\ast$}
        \put (61,65) {\small$\overline\beta$}
        \put (59,12) {\small$\overline \delta$}
        \put (64,35) {\small $j$}}
        \rule[-3.2em]{0em}{7em}
\end{align*}
where $a_k: k\rightarrow k^{\ast\ast}$ denotes the pivotal structure isomorphism. 

Since $C(Z(A))_2$ and $C(Z(\unit))_2$ are both modular invariant vectors and the space of invariants is one-dimensional, and since $C(Z(A))_2$ and $C(Z(\unit))_2$ are non-zero by Lemma~\ref{lem:CZ1-nonzero}
there exists $\lambda_2\in \mathbbm{k}^\times$ such that 
\begin{equation}\label{eq:genus2modinv}
C(Z(A))_2 ~=~ \lambda_2\, C(Z(\unit))_2~.
\end{equation}
Let $i,j,k\in I$ such that $N_{ij}^{~k}\neq 0$, then by comparing both sides of \eqref{eq:genus2modinv} and using Example~\ref{ex:Z(1)-const} we obtain 
\begin{equation}\label{eq:lam-lam-cond1}
\varepsilon_0\,\lambda_{i \ibar}^\unit \,(\lambda_{j}^{\ibar k})^\alpha_\beta \,(\lambda_{k}^{ij})^\gamma_\delta ~=~ \delta_{\alpha,\beta} \,\delta_{\gamma,\delta}\, \lambda_2~.
\end{equation}
For $\alpha=\beta$ and $\gamma=\delta$ we get
\begin{equation}
\varepsilon_0\,\lambda_{i\ibar}^\unit \,(\lambda_{j}^{\ibar k})^\alpha_\alpha \,(\lambda^{ij}_{k})^\gamma_\gamma 
~=~ 
\lambda_2 \neq 0 ~.
\end{equation}
This shows that $(\lambda_{j}^{\ibar k})^\alpha_\alpha$ and $(\lambda^{ij}_{k})^\gamma_\gamma$ are non-zero and independent of $\alpha$, $\gamma$. I.e.\ for $N_{ij}^{~k} \neq 0$ there exists
$\lambda^{ij}_k \in \mathbbm{k}^\times$ such that
$(\lambda^{ij}_{k})^\gamma_\gamma = \lambda^{ij}_k$ for $\gamma = 1,\dots,N_{ij}^{~k}$.
Taking $\alpha=\beta$ but
$\gamma \neq \delta$ in \eqref{eq:lam-lam-cond1} gives the desired form for the comultiplication structure constants
\begin{equation}\label{eq:lam-is-delta}
    (\lambda^{ij}_{k})^\gamma_\delta
    ~=~ \delta_{\gamma,\delta}\,\lambda^{ij}_k ~.
\end{equation}
To get also the expression for the structure constants of the product as claimed in the lemma, insert~\eqref{eq:ZA-constants} into $\mu = ((\varepsilon \circ \mu) \otimes \id{})\circ (\id{} \otimes \Delta)$. This gives 
\begin{equation}
(\lambda^{k}_{ij})^\beta_\alpha ~=~ \delta_{\alpha,\beta} \, \varepsilon_0\, \lambda_{i\ibar}^\unit\,\lambda_{j}^{\ibar k} ~=~ \delta_{\alpha,\beta} \, \lambda_{ij}^k~,
\end{equation}
with $\lambda_{ij}^k = \varepsilon_0\, \lambda_{i\ibar}^\unit\,\lambda_{j}^{\ibar k} \neq 0$.
\end{proof}

\medskip

\noindent
\begin{minipage}{\textwidth}
\begin{lemma}\label{lem:someproperties}
The structure constants obey the following properties: 
\begin{enumerate}
    \item (Unitality and counitality) $\lambda_{\unit i}^{i} = \lambda_{i\unit}^{i} = \eta_0^{-1}$ and $\lambda_{i}^{\unit i} = \lambda_{i}^{i\unit } = \varepsilon_0 ^{-1}$
    \item (Commutativity) $\lambda_{ij}^{k} = \lambda_{ji}^{k}$
    \item (Index lowering and raising)
    $\lambda^{k}_{ij} =\, \varepsilon_0\, \lambda_{i\ibar}^\unit\,\lambda_{j}^{\ibar k}$
    and 
    $\lambda_{k}^{ij} = \eta_0 \, \lambda_{\unit }^{i\ibar }\, \lambda_{\ibar k}^{j}$ 
\end{enumerate}
\end{lemma}
\end{minipage}

\begin{proof}
Part 1 follows directly by evaluating the unitality and the counitality conditions for the morphisms in \eqref{eq:ZA-constants}. 

For part 2, 
let $i,j,k \in I$ with $N_{ij}^{~k}\neq 0$. From Lemma~\ref{lem:diagonal} we get 
    \begin{equation}
        r_k\circ \mu_{Z(A)}\circ (e_i\otimes e_j) = \lambda_{ij}^k\, r_k\circ \mu_{Z(\unit)}\circ (e_i\otimes e_j) ~.
    \end{equation}
Composing both sides of this equation with the braiding
$c_{j,i} : j \otimes i \to  i \otimes j$ 
and using naturality, we get
    \begin{align}
        r_k\circ \mu_{Z(A)} \circ c_{Z(A),Z(A)}\circ (e_j\otimes e_i) &= \lambda_{ij}^k\, r_k\circ \mu_{Z(\unit)} \circ c_{Z(\unit), Z(\unit)}\circ (e_j\otimes e_i)
        \nonumber\\ 
        &=\lambda_{ij}^k\, r_k\circ \mu_{Z(\unit)}\circ (e_j\otimes e_i)~.
    \end{align}
In the last step we used the commutativity of $Z(\unit)$. Making use of commutativity of $Z(A)$, \ie $\mu_{Z(A)}\circ c_{Z(A),Z(A)} = \mu_{Z(A)}$, finally implies part 2. 

In part 3, the first equality was already given at the end of the previous proof, and the second follows analogously by inserting \eqref{eq:ZA-constants} into 
    $\Delta = (\id{} \otimes \mu) \circ ((\Delta \circ \eta) \otimes \id{})$.
\end{proof}

\subsection{Sequence of isomorphisms}\label{sec:isosequence}

Given an object automorphism $f \in \Aut(Z(A))$, one can give an isomorphic (haploid commutative symmetric normalised-special modular invariant) Frobenius algebra 
    $\tilde Z \equiv f_\ast(Z(A))$.
Its underlying object is again $Z(A)$ but its structure morphisms are
\begin{align}
    \tilde{\mu} &= f\circ \mu_{Z(A)} \circ (f^{-1}\otimes f^{-1})~,  & \tilde{\eta} &= f \circ \eta_{Z(A)} ~,\nonumber\\
    \tilde{\Delta} &= (f\otimes f)\circ \Delta_{Z(A)} \circ f^{-1}~, & \tilde{\varepsilon} &= \varepsilon_{Z(A)} \circ f^{-1}~.
\end{align}
This is the unique Frobenius algebra structure such that $f : Z(A) \to \tilde Z$ is an isomorphism of Frobenius algebras.

The isomorphism $f$ is determined by invertible scalars $\{f_i\}$ as the underlying object of $Z(A)$ is $\bigoplus_{i\in I}i^\ast\times i$:
\begin{equation}
    f = \sum_{i\in I} f_i \, e_i \circ r_i ~.
\end{equation}
The new structure constants are then given by 
\begin{equation}
    \Tilde{\lambda}_{ij}^k = \frac{f_k}{f_i f_j}\lambda_{ij}^k
    ~~,\quad
        \Tilde{\lambda}_{k}^{ij} = \frac{f_i f_j}{f_k}\lambda_{k}^{ij}
    ~~,\quad
    \tilde\eta_0 = f_\unit \, \eta_0
    ~~,\quad
    \tilde\varepsilon_0 = f_\unit^{-1} \, \varepsilon_0
    ~.
\end{equation}
The new constants defined as above still obey the equations of Lemmata \ref{lem:someproperties} and \ref{lem:lambdaprop}.

We will find a sequence of such Frobenius algebra isomorphisms that take $Z(A)$ into $Z(\unit)$.  
In other words, we need to find (a sequence of) transformations $f_i$ such that $\tilde\eta_0 = \tilde\varepsilon_0=1$ and $\Tilde{\lambda}_{ij}^k =\Tilde{\lambda}^{ij}_k=1$ whenever 
$N_{ij}^{~k}\neq 0$. 

\subsubsection*{First Step}

Our first step will be to normalise the constants $\lambda_{i\unit }^i,\lambda_{\unit i}^i$ and $\lambda_{i\overline{i}}^\unit $. We do this by fixing $f_i$ such that $f_{\overline{i}}f_i = \lambda_{\unit \unit }^\unit  \lambda_{i\overline{i}}^\unit $ (for instance pick any square root $f_i = f_{\ibar} = \sqrt{\lambda_{\unit \unit }^\unit  \lambda_{i\ibar}^\unit }$) and fix $f_\unit = \lambda_{\unit \unit}^{\unit}$.
For example,
\begin{equation}
    \tilde\lambda_{i\overline{i}}^\unit 
    ~=~
    \frac{f_\unit}{f_i\,f_{\overline{i}}} \,\lambda_{i\overline{i}}^\unit 
    ~=~
    \frac{\lambda_{\unit \unit}^{\unit}}{\lambda_{\unit \unit }^\unit  \lambda_{i\overline{i}}^\unit} \,\lambda_{i\overline{i}}^\unit 
     ~=~ 1 ~.
\end{equation}
Assuming we applied this isomorphism, we may now
start with constants such that $\lambda_{i\unit }^i=\lambda_{\unit i}^i =1= \lambda_{i\overline{i}}^\unit $ for all $i\in I$. 
By Lemma~\ref{lem:someproperties}, this implies $\eta_0=\varepsilon_0=1$, as well as
\begin{equation}
\lambda_{k}^{ij} = \lambda_{\ibar k}^{j} ~.
\end{equation}
\ie we can raise or lower indices by conjugating the respective label. To avoid confusion, we will denote this algebra by $Z$, which is isomorphic to $Z(A)$ as a Frobenius algebra.

The above conditions on $\lambda_{ij}^k$, $\lambda^{ij}_k$, $\eta_0$, $\varepsilon_0$ are preserved by isomorphisms $f$ that satisfy
\begin{equation}\label{eq:fixed-step1}
    f_\unit = 1 \quad\text{and}\quad
    f_i f_{\ibar}= 1 ~~\text{for all $i \in I$} ~.
\end{equation}

\subsubsection*{Second Step}

To exploit the irreducibility of
the $V_g^{\catname C}$
for higher genus, it is convenient to introduce the notion of an $I$-fusion tree. 

\begin{definition}
\begin{itemize}
    \item A \textit{3-valent tree} is a tree graph, where each vertex is 3-valent with one incoming edge and two outgoing edges. 
    \item An \textit{$I$-fusion tree} is a 3-valent tree such that each edge is labelled by an element in $I$, and such that at each vertex $v$ the following condition is satisfied: if the incoming edge at $v$ is labelled $k$ and the two outgoing edges at $v$ are labelled $i$, $j$, then $N_{ij}^{~k} \neq 0$. 
    
    The outgoing edges of an $I$-fusion tree are ordered (we will label them $1,\dots,m$).

    \item Let $i,j_1,\dots, j_m \in I$. An $(i;j_1,\dots,j_m)$-fusion tree is an $I$-fusion tree such that the incoming edge is labelled by $i$ and the outgoing edges are labelled by $j_1,\dots, j_m$.
\end{itemize}    
\end{definition}

We stress that an $I$-fusion tree $\Omega$ is \textsl{not} a string diagram in $\catname C$. Namely, $\Omega$ only records labels in $I$ and does not include a specific morphism at each vertex.

Let $Z$ be a Frobenius algebra isomorphic to $Z(A)$ as a Frobenius algebra, and with structure constants $\lambda^{ij}_k$, etc., normalised as in the first step.
To a vertex $v$ of an $I$-fusion tree with incoming label $k$ and outgoing labels $i$, $j$ we assign  the number $\lambda(v) := \lambda_k^{ij}$. To the whole $I$-fusion tree we assign the product of the structure constants at each vertex,
\begin{equation}\label{eq:lam-def}
\lambda: \{I\mbox{-}\text{fusion trees}\} \longrightarrow \mathbbm{k}^\times
~~,\quad
\Omega 
    ~\longmapsto~
\lambda(\Omega) = \prod_{v\text{ vertex}} \lambda(v)~  .
\end{equation}

\begin{lemma}\label{lem:lambdaprop}
Let $i_1,\dots, i_g \in I$ and $\Omega$ be a $(\unit ;i_1,\ibar_1,\dots, i_g, \ibar_g)$-fusion tree.
Then
\begin{equation}\label{eq:lambdaprop}
      \lambda(\Omega) = 1~,
\end{equation}
independent of the choice of $i_1,\dots , i_g$ and $\Omega$. 
\end{lemma}
\begin{proof}
By irreducibility of the $V_g^{\catname C}$ and
by Lemma~\ref{lem:CZ1-nonzero} there is a $\lambda_g\in \mathbbm{k}^\times$ such that 
\begin{equation}\label{eq:lambdag}
C(Z)_g = \lambda_g\, C(Z(\unit))_g~.
\end{equation}
Fix a 3-valent tree $\Gamma$ with $2g$ leaves. By decorating each vertex with the coproduct, each such tree gives a realisation of the iterated coproduct $\Delta^{(2g)} : Z \to Z^{\otimes 2g}$. Using labellings of $\Gamma$ by $I$, we get a direct sum decomposition
\begin{equation}
\Delta_Z^{(2g)} \circ \eta_Z = \bigoplus_{\Omega} \lambda(\Omega)\,D_\Omega~ . 
\end{equation}
Here, the direct sum runs over $I$-fusion trees $\Omega$ with underlying unlabelled tree $\Gamma$, where the unique incoming edge is labelled by $\unit$.
The factor $\lambda(\Omega)$ is the product of structure constants as defined in \eqref{eq:lam-def}.
$D_\Omega$ is the summand of $\Delta_Z^{(2g)} \circ \eta_Z$ 
where for an edge labelled $k$ the corresponding tensor factor $Z$ is projected to $k^* \times k$.
The following example illustrates the procedure for $g=2$:
\begin{equation}
\Gamma = \boxpic{0.7}{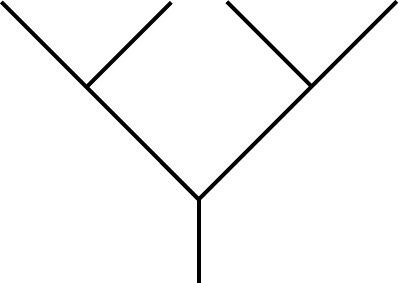}{}
\quad
\Omega = \boxpic{0.7}{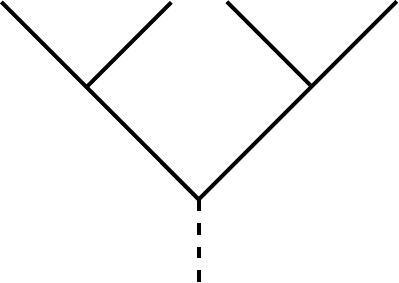}{
\put (0,75) {$i_1$}
\put (35,75) {$\ibar_1$}
\put (55,75) {$i_2$}
\put (95,75) {$\ibar_2$}
\put (25,25) {$k$}
\put (70,25) {$\kbar$}
}
\quad
D_\Omega = \boxpic{0.7}{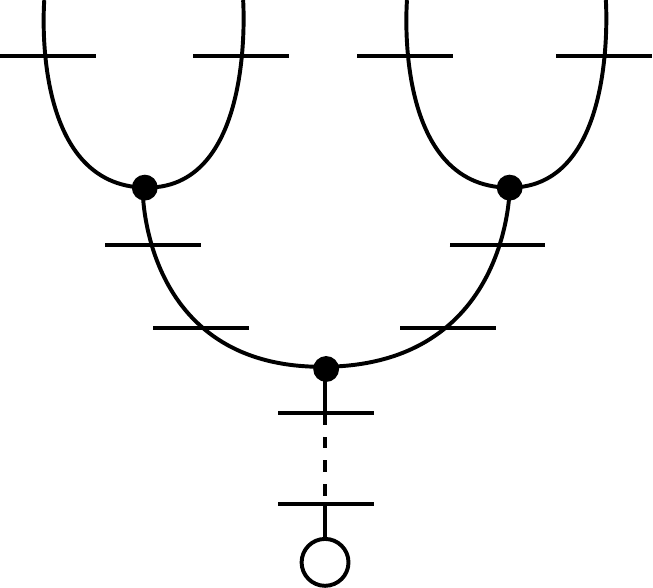}{
\put (-5,94) {\small $i_1^\ast\times i_1$}
\put (25,94) {\small $\ibar_1^\ast\times \ibar_1$}
\put (55,94) {\small $i_2^\ast\times i_2$}
\put (85,94) {\small $\ibar_2^\ast\times \ibar_2$}
\put (1,43) {\small $k^\ast\times k$}
\put (76,43) {\small $\kbar^\ast\times \kbar$}
\put (51,18) {\small $\unit\times \unit$}
\put (60,29) {\small $Z(\unit)$}
\put (34,60) {\small $Z(\unit)$}
\put (88,60) {\small $Z(\unit)$}
}
\end{equation}
Here, the coproduct is that of $Z(\unit)$ as given in \eqref{eq:Z1-product-coproduct}, for which all structure constants are $1$.

The important point to realise is that the $D_\Omega$ are linearly independent for the different choices of $\Omega$ (but for fixed $\Gamma$).
This can be seen for example by composing with the corresponding dual graph with in- and outgoing edges exchanged, which provides a non-degenerate pairing.

We can thus evaluate \eqref{eq:lambdag} summand by summand.
For $Z$ we get a factor $\lambda(\Omega)$ as in \eqref{eq:lam-def}, while for $Z(\unit)$ the structure constants are all $1$. Altogether we obtain, for all $I$-fusion trees $\Omega$ with underlying 3-valent tree $\Gamma$,
\begin{equation}
\lambda(\Omega)\,D_\Omega = \lambda_g\, D_\Omega~ . 
\end{equation}
Finally, to compute $\lambda_g$, take the $I$-fusion tree $\Omega$ where all edges are labelled by $\unit$. Since $\lambda_{\unit}^{\unit\unit}=1$, this results in $\lambda(\Omega)=1$.
\end{proof}

The next lemma is the first place where the universal grading group becomes important. Namely, the elements in the neutral component $R_{ad}$ are precisely those that ``can be created by handles''
(cf.\ Figure~\ref{fig:handlebodycoupon}). This property can be used to set the corresponding structure constants to $1$:

\begin{lemma}\label{lem:adjointsubring}
There exist $f_i$ satisfying~\eqref{eq:fixed-step1} such that $\Tilde{\lambda}^{ij}_k =1$ 
for all $i,j,k \in I_{ad}$ with $N_{ij}^{~k} \neq 0$.
\end{lemma}
\begin{proof}
Recall the filtration of the adjoint subring given in~\eqref{eq:filtration}. 
For $b_i \in R^{(n)}_{ad}$, $i \neq \unit$ there are
$m_1,\dots, m_{n}$ such that $b_i$ is contained $b_{m_{1}}b_{\overline m_{1}}\dots b_{m_{n}}b_{\overline m_{n}}$. In other words, 
there exists a $(i;m_1,\overline m_1, \dots, m_{n},\overline m_{n})$-fusion tree
\begin{equation}
\Omega_i ~=~
    \boxpic{0.7}{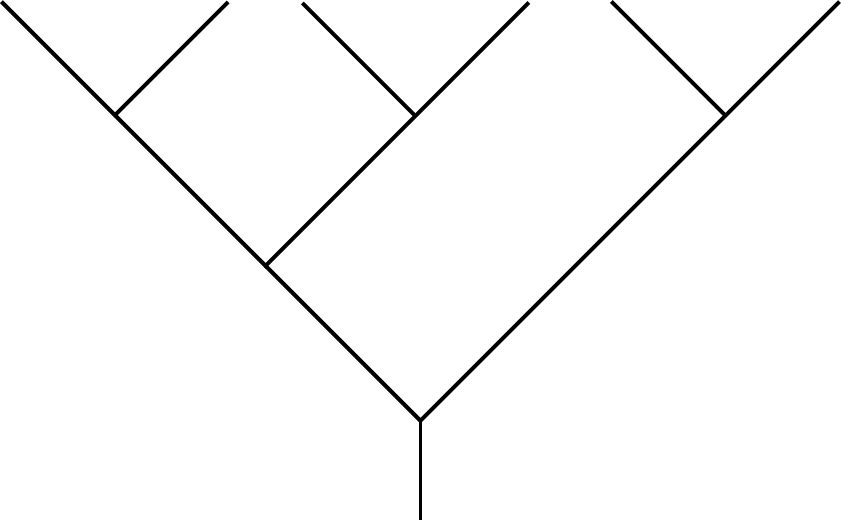}{
        \put (-5,63) {$m_1$}
        \put (22,63) {$\overline m_1$}
        \put (33,63) {$m_2$}
        \put (58,63) {$\overline m_2$}
        \put (70,63) {$m_{n}$}
        \put (98,63) {$\overline m_{n}$}
        \put (14,37) {$k_1$}
        \put (43,37) {$k_2$}
        \put (82,37) {$k_{n}$}
        \put (52,0) {$i$}
        \put (30,22) {$d_3$}
        \put (42,10) {$d_n$}
        \put (57,37) {$\cdots$}
        \put (36.5,15) {\small$\ddots$}
        }
        \rule[-4em]{0em}{10em}
        \quad .
    \end{equation}
We set $f_i = \lambda(\Omega_i)$. To check that the condition $f_i f_{\ibar} = 1$ in~\eqref{eq:fixed-step1} is satisfied, apply Lemma~\ref{lem:lambdaprop} to the fusion tree
\begin{equation}\label{eq:2omegas}
    \boxpic{0.7}{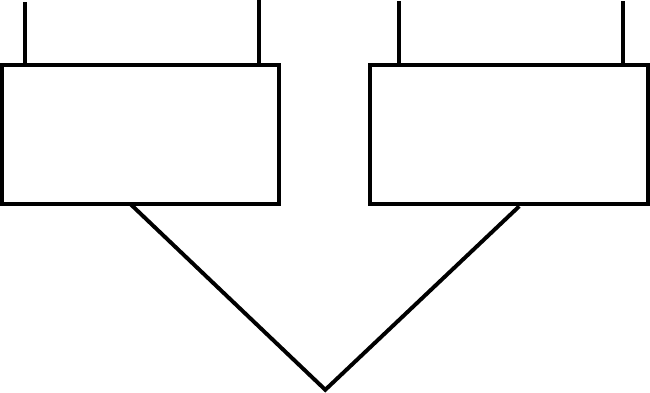}{
    \put (16,36) {$\Omega_i$}
    \put (74,36) {$\Omega_{\ibar}$}
    \put (16,55) {$\dots$}
    \put (74,55) {$\dots$}
    \put (28,9) {$i$}
    \put (70,9) {$\ibar$}}\quad. 
\end{equation}
For $i,j,k \in I_{ad}$ and their associated fusion trees $\Omega_i,\Omega_j$ and $\Omega_{\kbar}$, consider the graph
\begin{equation}\label{eq:3omegas}
    \boxpic{0.7}{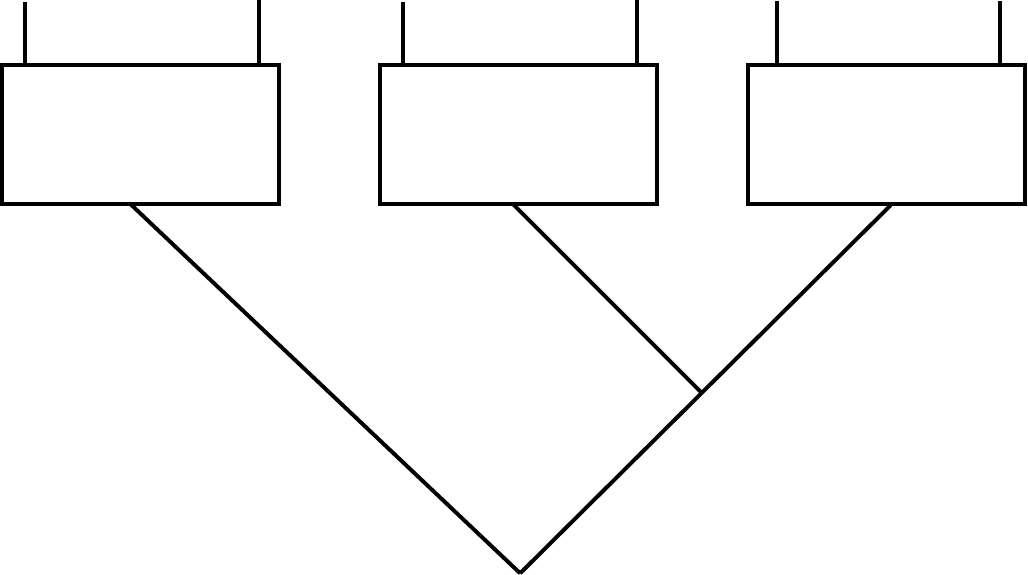}{
    \put (11,41) {$\Omega_{\kbar}$}
    \put (47,41) {$\Omega_i$}
    \put (83,41) {$\Omega_j$}
    \put (10,53) {$\dots$}
    \put (46,53) {$\dots$}
    \put (83,53) {$\dots$}
    \put (35,4) {$\kbar$}
    \put (60,4) {$k$}
    \put (57,22) {$i$}
    \put (77,22) {$j$}
    }\quad. 
\end{equation}
Lemma~\ref{lem:lambdaprop} gives the condition
\begin{equation}
    \lambda(\Omega_{\kbar}) \, \lambda(\Omega_i)\,\lambda(\Omega_j) \, \lambda^{ij}_k ~=~1 ~.
\end{equation}
Substituting $f_i = \lambda(\Omega_i)$ and recalling that $f_k f_{\kbar} = 1$ finally gives $\lambda_{k}^{ij} = \frac{f_k}{f_i f_j}$, \ie $\Tilde{\lambda}_{k}^{ij} = 1$.
\end{proof}

We will need to know how the $\lambda(\Omega_i)$ change in the new normalisation given by the $f_i$. We have 
\begin{equation}\label{eq:lam-Omega-renorm}
    \tilde\lambda(\Omega_i) ~=~ \frac{f_{m_1}f_{\overline m_1} \cdots f_{m_n} f_{\overline m_n}}{f_i} \lambda(\Omega_i) ~=~1 ~,
\end{equation}
since $f_m f_{\overline m}=1$ and $f_i = \lambda(\Omega_i)$.

Note that in the proof of Lemma~\ref{lem:adjointsubring} we have only used the irreducibility of $V_g^{\catname C}$ up to $g =3 N$, where $N$ was defined in Section~\ref{sec:universalgroup} to be the maximal degree in the filtration of $R_{ad}$. Below we will need to go up to $g=3N+2$, see Remark~\ref{rem:maximal-g}.

\medskip

Let $i,j,k \in I$ be such that $N_{ij}^{~k}\neq 0$. At this point we have achieved $\lambda^{ij}_k=1$ whenever at least one of $i,j,k$ is given by $\unit$ (Step 1), and $\lambda^{ij}_k=1$ for $i,j,k \in I_{ad}$ (Step 2). We are still free to choose all $f_i$ with $i \notin I_{ad}$, subject to \eqref{eq:fixed-step1}.
Recall that in the proof of Lemma~\ref{lem:adjointsubring} we fixed a fusion graph $\Omega_i$ for each $i \in I_{ad}$, and that by \eqref{eq:lam-Omega-renorm} we have in the new normalisation:
\begin{equation}\label{eq:lam-Omega-i-is-1}
    \lambda(\Omega_i) = 1 \quad \text{for} ~~ i \in I_{ad}~.
\end{equation}

\subsubsection*{Third Step}

The following lemma is an extension of Lemma~\ref{lem:lambdaprop} to allow any reordering of the outgoing labels. 
\begin{lemma}\label{lem:lambdaprop-perm}
Let $i_1,\dots, i_g \in I$, $\sigma \in S_{2g}$ and $\Omega$ be a $(\unit; (i_1,\ibar_1 \dots, i_g, \ibar_g).\sigma)$-fusion tree where the permutation $\sigma$ acts by changing the order of the $2g$
outgoing labels accordingly. Then $\lambda (\Omega) = 1$.
\end{lemma}
\begin{proof}
Let $\beta_{2g}$ be any $2g$-braid, whose underlying permutation is $\sigma \in S_{2g}$. Since $Z$ and $Z(\unit)$ are cocommutative, we have $\beta_{2g} \circ \Delta^{(2g)} = \Delta^{(2g)}$ for both of them. 

We proceed as in the proof of Lemma~\ref{lem:lambdaprop} by expressing $\Delta_Z\circ \eta_Z$ as a direct sum where such fusion trees appear. Using cocommutativity of $Z$ and $Z(\unit)$, we get a direct sum decomposition
\begin{equation}
\Delta_Z^{(2g)} \circ \eta_Z = \beta_{2g}\circ \Delta_Z^{(2g)} \circ \eta_Z = \bigoplus_{\Omega} \lambda(\Omega)\,\beta_{2g} \circ D_\Omega~ 
~~,\quad
\Delta_{Z(\unit)}^{(2g)} \circ \eta_{Z(\unit)} = 
\bigoplus_{\Omega} \beta_{2g} \circ D_\Omega~ .
\end{equation}
where the direct sum is over $I$-fusion trees $\Omega$. 
Next, insert this into the definition of $C(Z)_g$ and use $C(Z)_g = C(Z(\unit))$ as in the proof of Lemma~\ref{lem:lambdaprop}. Comparing linearly independent terms gives
\begin{equation}
    \lambda(\Omega) = 1
\end{equation}
for any $(\unit; (i_1,\ibar_1 \dots, i_g, \ibar_g).\sigma)$-fusion tree $\Omega$.
\end{proof}

Using Lemma~\ref{lem:lambdaprop-perm}, one can deduce from the fusion tree 
\begin{equation}\label{eq:lam-inv-fig}
    \Omega = \boxpic{0.7}{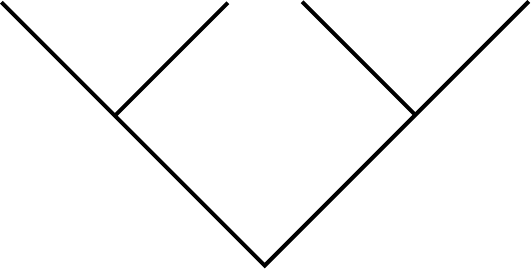}{
    \put (-2,52) {$i$}
    \put (38,52) {$j$}
    \put (54,52) {$\ibar$}
    \put (95,52) {$\jbar$}
    \put (28,9) {$k$}
    \put (70,9) {$\kbar$}} 
\end{equation}
the equality 
\begin{equation}\label{eq:lam-inv}
    \lambda_{\kbar}^{\ibar\jbar} = (\lambda_k^{ij})^{-1}~. 
\end{equation}

\begin{lemma}\label{lem:adj-g-lamprop}
Let $k,l \in I_g$ and $i,j \in I_{ad}$ such that $N_{ik}^l, N_{jk}^l \neq 0$. Then, $\lambda_{l}^{ik} = \lambda_{l}^{jk}$.
\end{lemma}
\begin{proof}
Recall the fusion trees $\Omega_i$ we picked for each $i \in I_{ad}$ in Step 2. 
Consider the fusion tree
\begin{equation}
\boxpic{0.7}{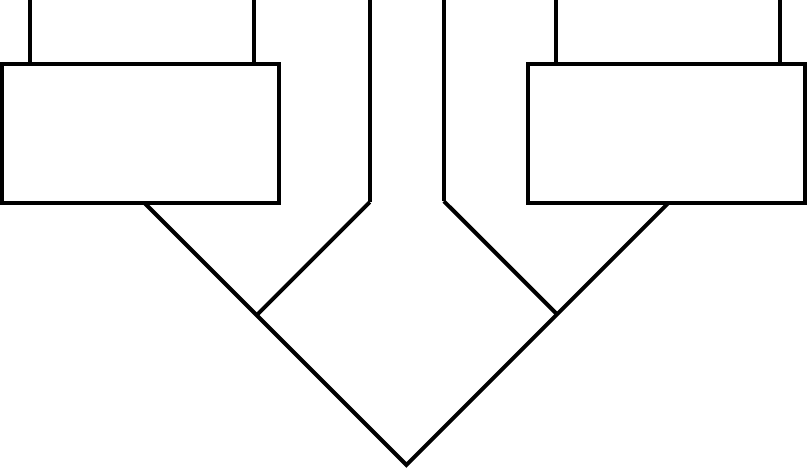}{
    \put (14,39) {$\Omega_i$}
    \put (79,39) {$\Omega_{\jbar}$}
    \put (13,55) {$\dots$}
    \put (78,55) {$\dots$}
    \put (22,20) {$i$}
    \put (77,20) {$\jbar$}
    \put (39,20) {$k$}
    \put (57,20) {$\kbar$}
    \put (37,5) {$l$}
    \put (63,5) {$\overline l$}}
\end{equation}
By Lemma~\ref{lem:lambdaprop} this gives the identity 
$\lambda(\Omega_i) \,\lambda(\Omega_{\jbar}) \, \lambda_{l}^{ik} \, \lambda_{\overline l}^{\kbar \jbar }=1$. 
Together with \eqref{eq:lam-Omega-i-is-1} and \eqref{eq:lam-inv}, we conclude $\lambda_{l}^{ik} = \lambda_{l}^{jk}$.
\end{proof}

\begin{lemma}
There exist $f_i$ satisfying \eqref{eq:fixed-step1} for all $i \in I$ and $f_i =1$ for $i\in I_{ad}$, such that $\Tilde{\lambda}_{k'}^{ik}= 1$ if $i\in I_{ad}$ and $N_{ik}^{~k'} \neq 0$. 
\end{lemma}

\begin{proof}
For each $g\in G$, fix an element $k_g \in I_g$,
such that $k_e = \unit$. By transitivity there exists some $i_g\in I_{ad}$ such that $N_{i_g \kbar_g}^{k_{g^{-1}}}\neq 0$. Then, find and fix $f_{k_g}$ for every $g$ such that
\begin{equation}\label{eq:fk_g}
    f_{k_g}f_{k_{g^{-1}}}  = \lambda_{k_{g^{-1}}}^{i_g\kbar_g}~
\end{equation}
and such that $f_\unit = 1$.
For instance $f_{k_g}= f_{k_{g^{-1}}} = \big(\lambda_{k_{g^{-1}}}^{i_g\kbar_g}\big)^{\frac12}$ is a consistent choice, since $\lambda_{k_{g}}^{i_g\kbar_{g^{-1}}} = \lambda_{k_{g^{-1}}}^{i_g \kbar_g}$ by raising and lowering indices (Lemma~\ref{lem:someproperties}). Note that this choice satisfies $f_\unit = 1$. 
Let $k\in I_g$ and $i\in I_{ad}$ be such that $N_{ik_g}^{k}\neq 0$ and define
\begin{equation}\label{eq:fk-g}
    f_k = \lambda_{k}^{ik_g}f_{k_g}~.
\end{equation}
This
is independent of the choice of $i$ by Lemma~\ref{lem:adj-g-lamprop} and is consistent with $k= k_g$ as $\lambda_{k_g}^{ik_g} = 1$ (choose $i=\unit$).
    
Let $k,k' \in I_g$ and $i,i',j\in I_{ad}$ such that $N_{i' k'}^{~k_g}, N_{jk}^{~k'}, N_{ik_g}^{~k} \neq 0$. 
The fusion tree
\begin{equation}
\boxpic{0.7}{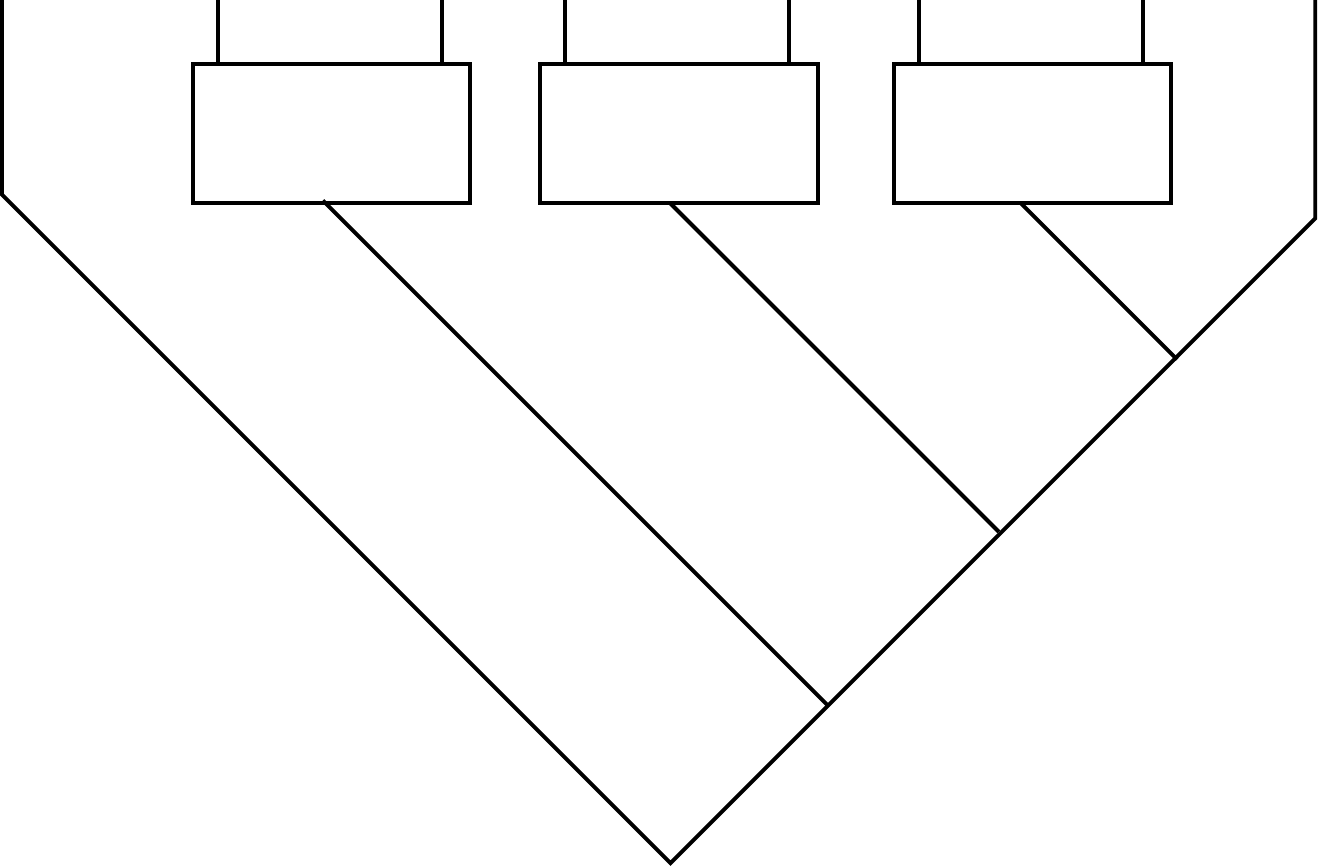}{
    \put (22,54) {$\Omega_{i'}$}
    \put (49,54) {$\Omega_{j}$}
    \put (75,54) {$\Omega_{i}$}
    \put (32,44) {$i'$}
    \put (59,44) {$j$}
    \put (85,44) {$i$}
    \put (22,63) {$\dots$}
    \put (48,63) {$\dots$}
    \put (75,63) {$\dots$}
    \put (-5,62) {$\kbar_g$}
    \put (56,2) {$k_g$}
    \put (70,15) {$k'$}
    \put (84,29) {$k$}
    \put (101,62) {$k_g$}}
\end{equation}
implies 
\begin{equation}\label{eq:f_g-fig}
    \lambda_{k}^{i k_g} \, \lambda_{k'}^{j k} \, \lambda_{k_g}^{i' k'} = 1~.
\end{equation}
From~\eqref{eq:lam-inv} and Lemma~\ref{lem:someproperties} we get $\lambda_{k_g}^{i' k'} = (\lambda_{k'}^{\ibar' k_g})^{-1}$. 
Inserting this in~\eqref{eq:f_g-fig} and using
\eqref{eq:fk-g} gives $\lambda_{k'}^{j k} = \frac{f_{k'}}{f_k}$. 
Furthermore, 
 \begin{equation}
     f_k f_{\kbar} 
     \overset{\eqref{eq:fk-g}}= \lambda_{k}^{i k_g}\lambda_{\kbar}^{jk_{g^{-1}}} f_{k_g} f_{k_{g^{-1}}} 
     \overset{\eqref{eq:fk_g}}=  \lambda_{k}^{i k_g}\lambda_{\kbar}^{jk_{g^{-1}}} \lambda_{k_{g^{-1}}}^{i_g \kbar_g} 
     ~=~ \lambda_{k}^{i k_g}\lambda_{\kbar_{g^{-1}}}^{jk} \lambda_{k_{g}}^{i_g \kbar_{g^{-1}}} 
     ~\overset{(*)}=~ 1~, 
 \end{equation}
 where $(*)$ follows from setting $k' = \kbar_{g^{-1}}$ in \eqref{eq:f_g-fig}.
\end{proof}

Therefore, in the new normalisation we have $\lambda_{k'}^{ik} = 1$ if $N_{ik}^{~k'}\neq 0$ and $i \in I_{ad}$.

\begin{lemma}\label{lem:lamgrading}
The structure constants depend only on the universal grading group, \ie for $i,i' \in I_g, j,j' \in I_h, k,k' \in I_{gh}$ with $N_{ij}^{~k},N_{i'j'}^{~k'}\neq 0$ we have
$\lambda^{ij}_k= \lambda^{i'j'}_{k'}$.
\end{lemma}

\begin{proof}
By transitivity of $R$, there exist $i_0, j_0,k_0 \in I_{ad}$ such that $N_{i_0 i}^{~i'},N_{j j_0}^{~j'}$ and $N_{k' k_0}^{~k}$ are non-zero. Consider the fusion tree
\begin{equation}
\boxpic{0.7}{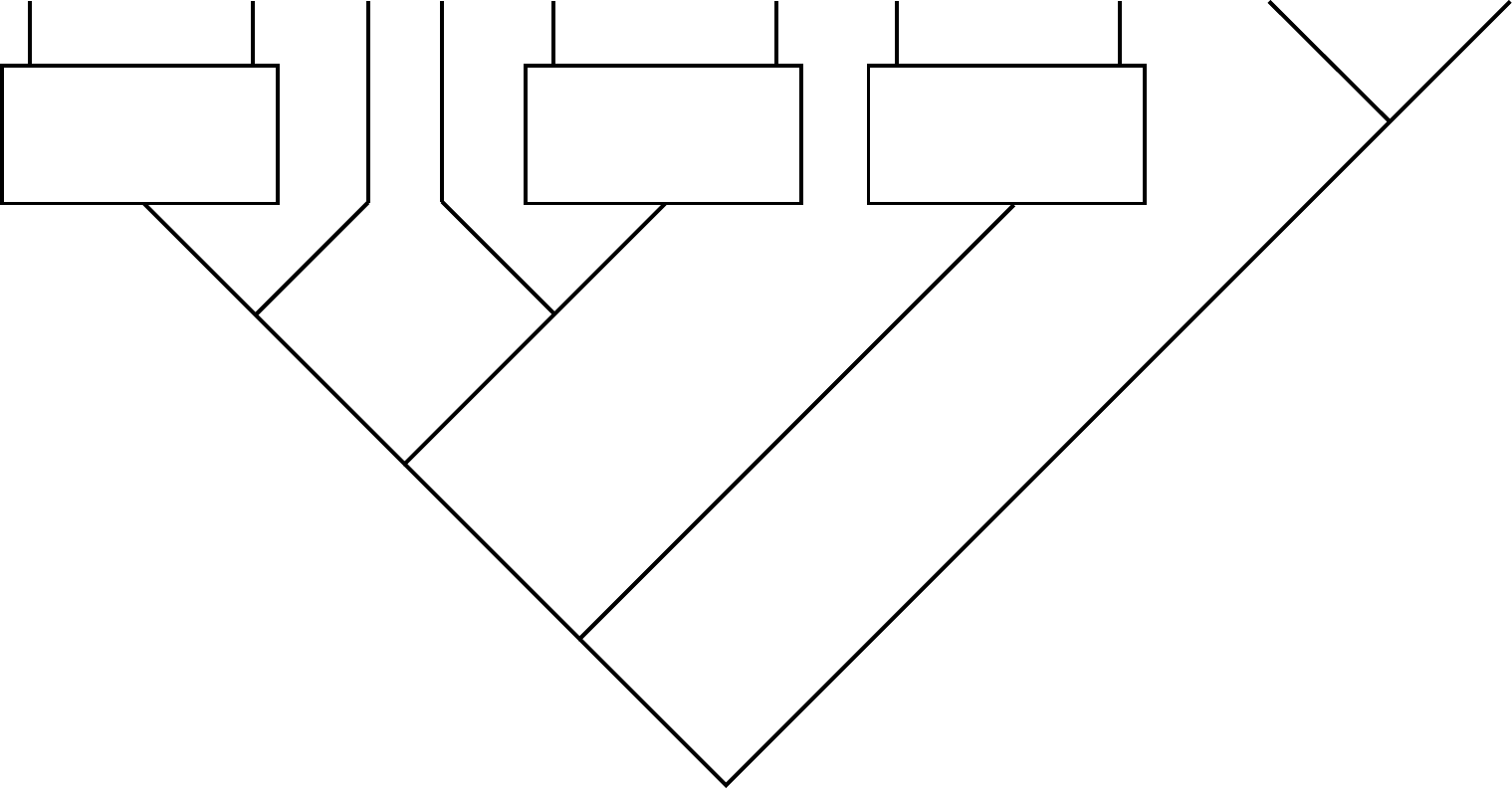}{
    \put (7,42) {$\Omega_{i_0}$}
    \put (42,42) {$\Omega_{j_0}$}
    \put (65,42) {$\Omega_{k_0}$}
    \put (6,50) {$\dots$}
    \put (41,50) {$\dots$}
    \put (64,50) {$\dots$}
    \put (22,50) {$i$}
    \put (30,50) {$j$}
    \put (41,2) {$k$}
    \put (31,23) {$j'$}
    \put (19,23) {$i'$}
    \put (29,13) {$k'$}
    \put (54,2) {$\kbar$}
    \put (82,50) {$\ibar$}
    \put (100,50) {$\jbar$}
    }
\end{equation}
which implies $\lambda^{i'j'}_{k'} \lambda^{\ibar \jbar}_{\kbar} = 1$
and so by \eqref{eq:lam-inv-fig} also
$\lambda^{i'j'}_{k'} = \lambda^{ij}_{k}$. 
\end{proof}

\begin{remark}\label{rem:maximal-g}
The proof of Lemma~\ref{lem:lamgrading} above is the place where the maximal genus $g$ occurs for which we use irreducibility of $V_g^{\catname C}$, namely $g= 3N+2$.
\end{remark}

The conditions on $\lambda^{ij}_k$ achieved up to this point are preserved by renormalisation constants $f_i$ which satisfy \eqref{eq:fixed-step1} as well as
\begin{equation}\label{eq:step3-f-conditions}
    f_i = 1 ~~\text{for all $i \in I_{ad}$} ~~,\quad
    f_i = f_j ~~\text{whenever $i,j \in I_g$ for some $g$} ~. 
\end{equation}

\subsubsection*{Final Step}

To conclude the proof, we will use group cohomology for the universal grading group $G$.
Namely, we define a 2-cochain $\omega : G \times G \to \mathbbm{k}^\times$ as follows. Given $g,h \in G$, pick $b_i \in R_g$, $b_j \in R_h$ as well as a $b_k \in R_{gh}$ that appears in the product $b_i b_j$. Then $N_{ij}^{~k} \neq 0$ and we set
\begin{equation}\label{eq:omega-lam-def}
    \omega(g,h) := \lambda^{ij}_k ~.
\end{equation}
By Lemma~\ref{lem:lamgrading}, this is independent of the choice of $i,j,k$.

\begin{lemma}\label{lem:2cocycle}
The 2-cochain $\omega$ is a symmetric normalised 2-cocycle. 
\end{lemma}
\begin{proof}
That $\omega$ is normalised, \ie that $\omega(e,g) = 1 = \omega(g,e)$, is just the normalisation condition $\lambda^{\unit i}_i = 1 = \lambda^{\unit i}_i$ achieved in step 1.
Symmetry of $\omega$, that is $\omega(g,h) = \omega(h,g)$ follows from the commutativity property of $\lambda^{ij}_k$ in Lemma~\ref{lem:someproperties}.

To show the  cocycle condition we will use coassociativity of the algebra $Z$. Given $f,g,h \in G$, pick $b_i \in R_f$, $b_j \in R_g$, and $b_k \in R_h$. 
Then choose $l\in I$ such that $b_l$ is a summand in the product $b_i b_j b_k$. This implies that $b_l \in R_{fgh}$.

In terms of structure constants, one side of the coassociativity condition for $Z$ can be rewritten as
\begin{align}
&
(r_i\otimes r_j\otimes r_k)
\circ 
(\Delta_Z\otimes \id{})
\circ
\Delta_Z
\circ
e_l 
\nonumber
\\
&\overset{(1)}=~
\sum_{p \in I}
(r_i\otimes r_j\otimes r_k)
\circ 
(\Delta_Z\otimes \id{})
\circ
((e_p \circ r_p) \otimes \id{}) 
\circ 
\Delta_Z
\circ
e_l 
\nonumber
\\
&\overset{(2)}=~
\sum_{p \in I} \lambda^{ij}_p \, \lambda^{pk}_l \,
(r_i\otimes r_j\otimes r_k)
\circ 
(\Delta_{Z(\unit)}\otimes \id{})
\circ
((e_p \circ r_p) \otimes \id{}) 
\circ 
\Delta_{Z(\unit)}
\circ
e_l 
\nonumber
\\
&\overset{(3)}=~\omega(f,g)\,\omega(fg,h)
\sum_{p \in I} 
(r_i\otimes r_j\otimes r_k)
\circ 
(\Delta_{Z(\unit)}\otimes \id{})
\circ
((e_p \circ r_p) \otimes \id{}) 
\circ 
\Delta_{Z(\unit)}
\circ
e_l 
\nonumber
\\
&\overset{(4)}=~\omega(f,g)\,\omega(fg,h)~
(r_i\otimes r_j\otimes r_k)
\circ 
(\Delta_{Z(\unit)}\otimes \id{})
\circ
\Delta_{Z(\unit)}
\circ
e_l 
\end{align}
In step 1 we expanded $\id{Z}$ into a direct sum over its component simple summands. This allows us in step 2 to insert the factors of $\lambda$ which give the difference between $\Delta_Z$ and $\Delta_{Z(\unit)}$ in each simple summand. (In this expression we take $\lambda$'s to be zero if their indices are not allowed by fusion.) The key step is equality 3. Here one uses that by the properties of the universal grading group, all $p \in I$ which give a nonzero contribution must have $b_p \in R_{fg}$, for else $N_{ij}^{~p}=0$. Thus, if we replace $\lambda$ by $\omega$ via \eqref{eq:omega-lam-def}, the prefactor becomes independent of $p$ and can be taking out of the sum. The sum over $p$ can then be carried out giving the result of step 4.

An analogous computation for the other side of the coassociativity condition for $Z$ gives
\begin{align}
&(r_i\otimes r_j\otimes r_k)
\circ 
(\id{} \otimes \Delta_Z)
\circ
\Delta_Z
\circ
e_l 
\nonumber
\\
&=~
\omega(g,h)\,\omega(f,gh)~
(r_i\otimes r_j\otimes r_k)
\circ 
(\id{} \otimes \Delta_{Z(\unit)})
\circ
\Delta_{Z(\unit)}
\circ
e_l  ~.
\end{align}
Comparing the two expressions and using coassociativity of $Z$ and $Z(\unit)$ results in 
\begin{equation}
    \omega(f,g) \, \omega(fg,h)~=~\omega(g,h)\,\omega(f,gh)~,
\end{equation}
which is the cocycle condition.
\end{proof}

In group cohomology there is a short exact sequence 
\begin{equation}\label{eq:ses}
0 \rightarrow \mathrm{Ext}(G,\mathbbm{k}^\times) \rightarrow H^2(G,\mathbbm{k}^\times) \rightarrow \Hom(\Lambda^2 G, \mathbbm{k}^\times) \rightarrow 0~,
\end{equation}
see \cite[Exercise V.6.5]{brown}. 
However, $\mathrm{Ext}(G,\mathbbm{k}^\times) = 0$ (as $\mathbbm{k}$ is algebraically closed, $\mathbbm{k}^\times$ is a divisible group, and so injective as an abelian group).
The second map in \eqref{eq:ses} is given by
\begin{equation}
    \psi ~\longmapsto~ \Big(\, g \wedge h \,\mapsto\, \frac{\psi(g,h)}{\psi(h,g)} \,\Big)~,
\end{equation}
and so any symmetric 2-cocycle is a coboundary. 

In particular, by Lemma~\ref{lem:2cocycle} $\omega$ is a coboundary, that is, there exist $\gamma_g \in \mathbbm{k}^\times$ such that
\begin{equation}
    \omega(g,h) ~=~ \frac{\gamma_g\,\gamma_h}{\gamma_{gh}}~.
\end{equation}
As $\omega$ is normalised, we have $\gamma_e=1$. Now choose $f_i = \gamma_g^{-1}$ 
whenever $i \in I_g$. 
This choice satisfies the conditions in \eqref{eq:step3-f-conditions}.
To see that also \eqref{eq:fixed-step1} holds, note that $f_i f_{\ibar} = (\gamma_g \gamma_{g^{-1}})^{-1} = \omega(g,g^{-1})^{-1}$. But by \eqref{eq:omega-lam-def} we have $\omega(g,g^{-1}) = \lambda^{i \ibar}_\unit = 1$, by step 1.
This finally gives
\begin{equation}
    \lambda^{ij}_k ~=~ \frac{f_k}{f_i\,f_j}~.
\end{equation}
We have now completed the proof that $Z(A) \cong Z(\unit)$ as algebras and thereby the proof of Theorem~\ref{thm:maintheorem}.

\appendix

\section{Dimension of a simple non-degenerate algebra}
\label{app:simple-lemma}

A pivotal category is spherical if its left and right traces are equal. A ribbon category is automatically spherical.
A multifusion category is the same as a fusion category, except that the tensor unit is not required to be simple. We refer to \cite[Ch.\,4]{EGNO} for more details.

The following more general statement implies part 3 of Lemma~\ref{lem:nondeg-ssFa}.

\begin{lemma}\label{lem:simple-sFa-dim}
Let $\catname F$ be a spherical multifusion category over an algebraically closed field $\mathbbm{k}$ (of any characteristic) and let $A \in \catname F$ be a simple $\Delta$-separable symmetric Frobenius algebra. Then $\dim_{\catname F}(A) \neq 0$.
\end{lemma}
\begin{proof}
To avoid cumbersome notation, in this proof we assume $\catname F$ to be strict.
Let $A_\mathrm{top} := \catname F(\unit, A)$ denote the \textit{topological algebra} of $A$, cf.\ \cite[Sec.\,3.4]{FRS1}. It is an algebra over $\mathbbm{k}$ via the product and unit
\begin{equation}
    \mu_\mathrm{top}(x,y) := \mu \circ (x \otimes y)\quad , \quad 
    1_\mathrm{top} := \eta ~.
\end{equation}
Define a pairing on $A_\mathrm{top}$ by 
\begin{equation}\label{eq:Atop-pairing}
    \langle x, y\rangle := \varepsilon \circ \mu_\mathrm{top}(x,y)~.
\end{equation}
Let $x\neq 0$ be an element in $A_\mathrm{top}$. The non-degeneracy of $A$ (cf.\ Remark~\ref{rem:non-deg-vs-Frob}) implies that $\Phi \circ x \neq 0$. Since $\catname F$ is semisimple, there is a $\varphi: A^\ast \rightarrow \unit $ such that $\varphi\circ \Phi\circ x \neq 0$. Using the expression in~\eqref{eq:symmetric} for $\Phi$, it follows for $y = \big[ \unit \xrightarrow{\operatorname{coev}_A} A \otimes A^* \xrightarrow{\id{A}\otimes \varphi} A \big]$ that
\begin{equation}
    \langle x,y\rangle = \varphi \circ \Phi \circ x \neq 0~.
\end{equation}
Therefore, the pairing on $A_\mathrm{top}$ is non-degenerate. 
Consider now the linear map $p: A_\mathrm{top}\rightarrow A_\mathrm{top}$ defined by 
\begin{equation}
    p(x) = \boxpic{0.7}{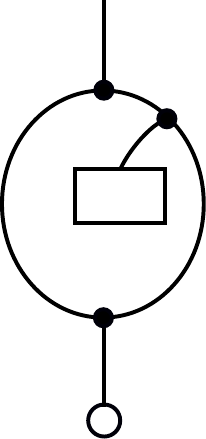}{\put (24,52.5) {$x$}} \quad .
\end{equation}
Since $A$ is $\Delta$-separable, it follows that $p^2 = p$ and $p(\eta) = \eta$. The pairing satisfies the following invariance property:
\begin{equation}
    \langle p(x),y \rangle \overset{(1)}{=} 
    \boxpic{0.7}{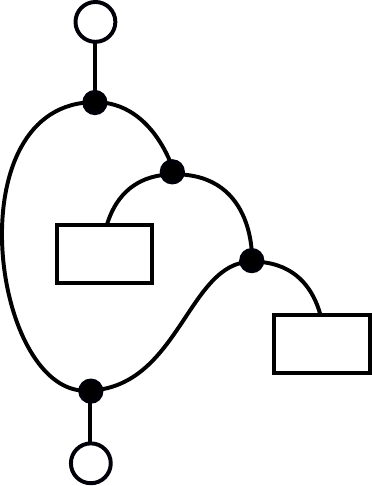}{
    \put (18,45) {$x$}
    \put (63,27) {$y$}
    }
    \overset{(2)}{=} 
    \boxpic{0.7}{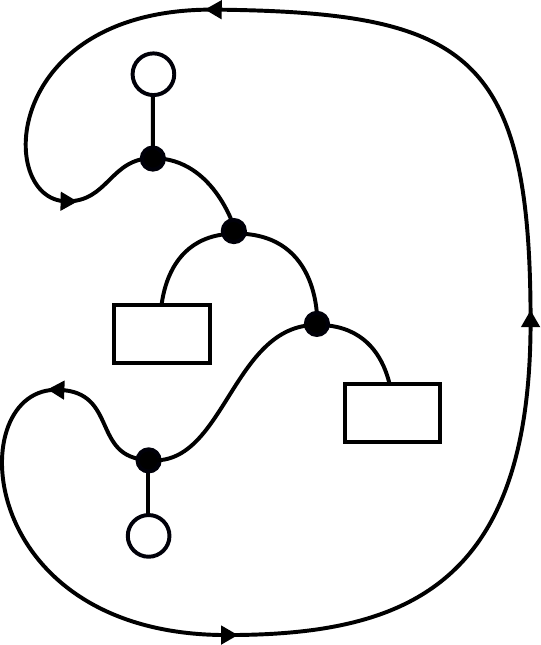}{
    \put (22,46) {$x$}
    \put (58,35) {$y$}
    }
    \overset{(3)}{=}
    \langle x, p(y) \rangle~,
\end{equation}
where (1) uses the associativity of $A$, (2) is immediate from sphericality of $\catname F$ and (3) can be verified using that $A$ is a symmetric Frobenius algebra. 

Write ${}_A\catname F_A(A,A)$ for the subspace of $A$-$A$-bimodule morphisms in $\catname F(A,A)$. 
Consider the linear map $\psi: {}_A\catname F_A (A,A)\rightarrow A_\mathrm{top}$ 
given by $\psi (f) := f \circ \eta$. 
It satisfies
\begin{equation}
    \psi (f) = f\circ \eta = f \circ p(\eta) = p(f\circ \eta) = p (\psi (f))~, 
\end{equation}
where we used that $f$ is a bimodule morphism to exchange $p$ with $f$. Hence, we have $\Im(\psi ) \subset \Im (p )$. 
Conversely, let $x \in A_\mathrm{top}$ and define $f_x \in \catname F(A,A)$ by 
\begin{equation}
    f_x = \boxpic{0.7}{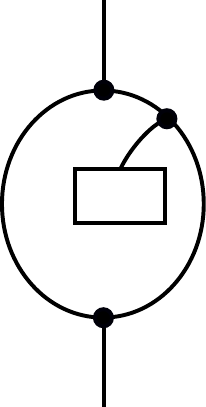}{\put (25,48.5) {$x$}}  \quad .
\end{equation}
One checks that $f_x \in {}_A \catname F_A (A,A)$ and $\psi(f_x) = p(x)$, so that $\Im (p)\subset \Im (\psi)$, \ie altogether $\Im (p)= \Im (\psi)$.

Since the algebra $A$ is simple, ${}_A \catname F_A (A,A) = \mathbbm{k}\id{}$
(this uses that $\mathbbm{k}$ is algebraically closed).
Therefore, $\dim \Im (p) \le 1$ and so in fact we have $\Im(p) = \mathbbm{k} \eta$.
By non-degeneracy of the pairing 
    in~\eqref{eq:Atop-pairing}
we can find some $y$ such that $\langle \eta,y \rangle \neq 0$. As $\eta$ is a basis for $\Im(p)$ there is $\lambda \in \mathbbm{k}$ with $p(y) = \lambda\,\eta$. Using this, we compute
\begin{equation}
    0 \neq \langle \eta , y \rangle = \langle p(\eta) , y \rangle = \langle \eta , p(y) \rangle = \lambda \langle \eta , \eta \rangle~.
\end{equation} 
Therefore $\langle \eta , \eta \rangle \neq 0$. Finally,
$\dim_{\catname F}(A) = \varepsilon \circ \eta = \langle \eta, \eta \rangle \neq 0$.  
\end{proof}
The condition that $A$ is symmetric cannot be dropped from Lemma~\ref{lem:simple-sFa-dim}. For example, the two-dimensional Clifford algebra with one odd generator in $SVect$ is simple $\Delta$-separable Frobenius (but not symmetric) and has dimension zero.


\newcommand{\arxiv}[2]{\href{http://arXiv.org/abs/#1}{#2}}
\newcommand{\doi}[2]{\href{http://doi.org/#1}{#2}}

\end{document}